\documentclass[12pt]{amsart}
\makeatletter

\@addtoreset{equation}{section}
\makeatother
\usepackage[margin=3cm]{geometry}
\newenvironment{nouppercase}{%
  \renewcommand{\uppercasenonmath}[1]{}}{}
\usepackage{amsmath,amssymb,amsthm,color}
\usepackage{mathrsfs} 
\usepackage[T1]{fontenc}
\usepackage{ytableau}
\ytableausetup{boxsize=normal,aligntableaux=center}  
\usepackage{tikz}
\usepackage[all]{xy}
\allowdisplaybreaks[1]
\theoremstyle{definition}
\newtheorem{theorem}[equation]{Theorem}
 
\newtheorem{defi}[equation]{Definition} 
\newtheorem*{conj*}{Conjecture}
\newtheorem*{theorem*}{Theorem}
\newtheorem{remark}[equation]{Remark} 
\newtheorem{cor}[equation]{Corollary}

\newtheorem{lemma}[equation]{Lemma}
\newtheorem{exam}[equation]{Example}
\def\zetas{\zeta^{\star}}
\begin{document}

\title{Symmetric Schur multiple zeta functions}
\author{Maki Nakasuji}
\address{Department of Information and Communication Science, Faculty of Science, Sophia
University, 7-1 Kioi-cho, Chiyoda-ku, Tokyo, 102-8554, Japan}
\email{nakasuji@sophia.ac.jp}
\author[W. Takeda]{Wataru Takeda}
\address{Department of Applied Mathematics, Tokyo University of Science,
1-3 Kagurazaka, Shinjuku-ku, Tokyo 162-8601, Japan.}
\email{w.takeda@rs.tus.ac.jp}

\begin{nouppercase}
\maketitle
\end{nouppercase}
\begin{abstract}
We introduce the multiple zeta functions with structures similar to those of symmetric functions such as Schur $P$-, Schur $Q$-, symplectic and orthogonal functions in the representation theory.
We first consider their basic properties such as a domain of absolute convergence.
And then by restricting to the truncated multiple zeta functions, we obtain the pfaffian expression of the Schur $Q$-multiple zeta functions, the sum formula for Schur $P$- and Schur $Q$-multiple zeta functions, the determinant expressions of symplectic and orthogonal Schur multiple zeta functions under an assumption on variables.
Finally, we generalize those to the quasi-symmetric functions.
\end{abstract}

\section{Introduction}
The well-known Hall-Littlewood symmetric functions are a family of symmetric functions that depend on a parameter $t$:\\
For $\lambda=(\lambda_1, \lambda_2, \cdots, \lambda_r)$ being a partition such that
$\lambda_i\in {\mathbb Z}$, $\lambda_1\geq \lambda_2\geq \cdots \geq \lambda_{r}\geq 0$ and
${\pmb x}=(x_1, x_2, \cdots, x_r)$ being variables, 
\begin{equation}\label{HallLittlewood}
P_{\lambda}({\pmb x}; t)=\frac{1}{v_{\lambda}(t)}\sum_{\sigma\in {\mathfrak S}_r} \sigma 
\left({\pmb x}^{\lambda} 
\prod_{\substack{1\leq i<j\leq r}}\frac{x_i-tx_j}{x_i-x_j}\right),
\end{equation}
where $\displaystyle{v_{\lambda}(t)=\prod_{j\geq 0} \prod_{k=1}^{m_j}\frac{1-t^k}{1-t}}$ with $m_j=\#\{i ~|~ 1\leq i\leq r, \lambda_i=j\}$, 
${\mathfrak S}_r$ is a symmetric group of degree $r$, 
and ${\pmb x}^{\lambda}=x_1^{\lambda_1}\ldots x_r^{\lambda_r}$.
When $t=0$, it is the Schur polynomial which we write it by $s_{\lambda}({\pmb x})=P_{\lambda}({\pmb x}; 0)$. Schur polynomials are irreducible general linear characters and can be written by means of semi-standard Young tableau combinatorially. Mainly in representation theory, much research has been done on this function since its introduction by Schur (\cite{sc01}).
Gessel and Viennot (\cite{gv}), after Lindstr\"{o}m introduced the idea in the work of \cite{li} for purposes in different contexts, gave a general methodology interpretating determinants as certain configurations of non-crossing paths in an acyclic digraph related with the theory of Young tableaux and symmetric functions, in what follows we call this technique as the Gessel-Viennot lattice path procedure, and proved the determinant formula called Jacobi-Trudi identity for Schur functions by using it.  
When $t=-1$ in \eqref{HallLittlewood}, the function is called Schur $P$-function or $Q$-function written by
$P_{\lambda}({\pmb x})=P_{\lambda}({\pmb x}; -1)$
or $Q_{\lambda}({\pmb x})=2^{r}P_{\lambda}({\pmb x}; -1)$, respectively,
which was also introduced by Schur (\cite{sc11}). We note that the Schur $Q$-function was originally defined via certain pfaffians in his analysis of projective representations of symmetric groups. The tableau description of Schur $Q$-functions was introduced by Stembridge (\cite{st89}) in using the theory of shifted tableaux developed by Worley (\cite{wo}) and Sagan (\cite{sa}), and the combinatorial structure of this function was revealed. In his paper \cite{st}, Stembridge obtained that the tableau definition agrees with Schur's pfaffian expressions by a generalization of the Gessel-Viennot lattice path procedure.
In parallel with the above theory, 
the symplectic and orthogonal Schur functions which are irreducible symplectic and orthogonal characters and can be also defined combinatorially, have been developed. It is well-known that
a similar discussion such as a determinant formula holds by using an analogue of the Gessel-Viennot lattice path procedure. (see Hamel-Goulden \cite{hg}, Hamel \cite{ha} and Foley-King \cite{fk}.)

The Schur multiple zeta function introduced by Nakasuji, Phuksuwan and Yamasaki (\cite{npy}),  is a generalization of both multiple zeta and zeta-star function of Euler-Zagier type
with a combinatorial structure similar to a Schur polynomial.
Since this function has a combinatorial and an analytic features,
the characteristics of both sides have been investigated in recent years. 
In \cite{npy}, they obtained some determinant formulas such as the Jacobi-Trudi formula, Giambelli formula and dual Cauchy formula for Schur multiple zeta functions by using the Gessel-Viennot lattice path procedure and properties of Young tableaux.
Nakasuji and Ohno (\cite{no}) obtained the duality formula and its generalization called Ohno relation for Schur multiple zeta functions which are the extension of those for multiple zeta functions of Euler-Zagier type. In the theory of original multiple zeta value, the Ohno relation includes some relations such as duality formula, sum formula (Theorem \ref{originalsum}) and Hoffman relation. However, for Schur multiple zeta functions, the sum formula has not been obtained from the Ohno relation.
Now it is natural to ask whether we can define multiple zeta functions with structures similar to those of symmetric functions such as Schur $P$- or $Q$-functions, symplectic or orthogonal functions.
In this paper, we will focus on this point.

In Section \ref{basicprop}, for ${ \pmb s}=(s_{ij})\in ST(\lambda, {\mathbb C})$ being the set of all shifted tableaux of shape $\lambda$ over ${\mathbb C}$,
we introduce the 
{\it Schur $P$-multiple zeta functions} and the {\it Schur $Q$-multiple zeta functions} 
of {shape} $\lambda$  as the following series
\[
\zeta_\lambda^P({ \pmb s})=\sum_{M\in PSST(\lambda)}\frac{1}{M^{ \pmb s}},
\text{ and }
    \zeta_\lambda^Q({ \pmb s})=\sum_{M\in QSST(\lambda)}\frac{1}{M^{ \pmb s}},
\]
respectively, 
where $PSST(\lambda)$ and $QSST(\lambda)$ are the sets of semi-standard marked shifted tableaux of shape $\lambda$ under certain conditions (see detail in Section \ref{basicprop}), and discuss their basic properties such as the domain of convergence.
In Section \ref{pfaffiansection}, we consider the pfaffian expression of the (truncated) Schur $Q$-multiple zeta functions by following Stembridge (\cite{st89}).
Here, the truncated Schur $Q$-multiple zeta function are
 \[
    \zeta_\lambda^{Q,N}({ \pmb s})=\sum_{M\in QSST_N(\lambda)}\frac{1}{M^{ \pmb s}}
\]
for a fixed positive integer $N\in \mathbb{N}$, where $QSST_N(\lambda)$ are the sets of all $(m_{ij})\in QSST(\lambda)$ such that $m_{ij}\le N$ for all $i,j$.
In Section \ref{skewsection}, we find out that the above discussion can be easily generalized to the skew type.
In Section \ref{outsidesec}, after reviewing the outside decomposition of shifted Young diagram according to Hamel-Goulden \cite{hg},
we apply it to our skew type Schur $Q$-multiple zeta functions and obtain the pfaffian expressions associated with that decomposition.
In Section \ref{sumformsec}, we will discuss the sum formula for our Schur $P$- and $Q$-multiple zeta functions.
Sections \ref{symplecticsection}, \ref{orthogonalsec}, and \ref{decompsec} are devoted to discussions of symplectic and orthogonal Schur multiple zeta functions.
Roughly speaking, for positive integer $N$ and ${ \pmb s}=(s_{ij})\in T(\lambda, {\mathbb C})$ being the set of all Young tableaux of shape $\lambda$ over ${\mathbb C}$, we define the 
{\it symplectic Schur multiple zeta functions} and the {\it orthogonal Schur multiple zeta functions} 
of {shape} $\lambda$ as the following series
\[
\zeta_\lambda^{{\rm sp},N}({ \pmb s})=\sum_{M\in SP_N(\lambda)}\frac{1}{M^{ \pmb s}},
\text{ and }
    \zeta_\lambda^{{\rm so},N}({ \pmb s})=\sum_{M\in SO_N(\lambda)}\frac{1}{M^{ \pmb s}},
\]
respectively, 
where $SP_N(\lambda)$ and $SO_N(\lambda)$ are the sets of all symplectic tableaux and so-tableaux of shape $\lambda$ (see detail in Section \ref{symplecticsection} and \ref{orthogonalsec}).
We compose directed graphs corresponding to those functions as analogous to the original symplectic and orthogonal Schur functions due to Hamel (\cite{ha}) and give the determinant expressions in a manner similar to Hamel where we apply the Stembridge Theorem \cite{st}.
Further, we will give their decomposition into a sum of truncated multiple zeta or zeta-star functions.
Lastly, in Section \ref{sec:Sqsf}, we study the extension of all of those functions to the quasi-symmetric functions. We obtain the pfaffian expressions for Schur $Q$-type quasi-symmetric functions and determinant expressions for symplectic type and orthogonal type quasi-symmetric functions.

\section{Basic properties of Schur $P$- and $Q$-multiple zeta functions}
\label{basicprop}
We first review the basic terminology to define Schur $P$- and $Q$-multiple zeta functions.
A partition $\lambda=(\lambda_1,\ldots,\lambda_r)$ is called strict, if $\lambda_1 > \lambda_2 > \cdots > \lambda_r \ge 0$. 
Then, we associate strict partition $\lambda$ with {\it the shifted diagram}
\[SD(\lambda)=\{(i, j)\in {\mathbb Z}^2 ~|~ 1\leq i\leq r, i\leq j\leq \lambda_i+i-1\}\] depicted as a collection of square boxes with the $i$-th row has $\lambda_i$ boxes.
{ We say that $(i,j)\in SD(\lambda)$ is a corner of $\lambda$ if $(i+1, j)\notin SD(\lambda)$ and $(i, j+1)\notin SD(\lambda)$
 and denote by $SC(\lambda) \subset SD(\lambda)$ the set of all corners of $\lambda$;
 for example, $SC((4,2,1))=\{(1,4),(3,3)\}$.}
For a strict partition $\lambda$, a shifted tableau $(t_{ij})$ of shape $\lambda$ over a set $X$ is a filling of $SD(\lambda)$ with $t_{ij}\in X$ into $(i,j)$ box of $SD(\lambda)$.
We denote by $ST(\lambda,X)$ the set of all shifted tableaux of shape $\lambda$ over $X$.

Let $\mathbb N'$ be the set $\{1',1,2',2,\ldots\}$ with the total ordering $1' < 1 < 2' < 2 <\cdots $. Then, a {\it semi-standard marked shifted tableau} $\pmb t=(t_{ij})\in ST(\lambda,\mathbb N')$ is obtained by numbering all the boxes of $SD(\lambda)$ with letters from $\mathbb N'$ such that
\begin{description}
    \item[PST1] the entries of $\pmb t$ are weakly increasing along each column and row of $\pmb t$,
    \item[PST2] for each $i=1,2,\ldots$, there is at most one $i'$ per row,
    \item[PST3] for each $i=1,2,\ldots$, there is at most one $i$ per column,
     \item[PST4] there is no $i'$ on the main diagonal.
\end{description}
We denote by $PSST(\lambda)$ the set of semi-standard marked shifted tableaux of shape $\lambda$.
Then for a given set ${ \pmb s}=(s_{ij})\in ST(\lambda,\mathbb{C})$ of variables, 
{\it Schur $P$-multiple zeta functions} of {shape} $\lambda$ are defined as
\begin{equation}
\label{def:P}
\zeta_\lambda^P({ \pmb s})=\sum_{M\in PSST(\lambda)}\frac{1}{M^{ \pmb s}},
\end{equation}
where $M^{ \pmb s}=\displaystyle{\prod_{(i, j)\in SD(\lambda)}|m_{ij}|^{s_{ij}}}$ for $M=(m_{ij})\in PSST(\lambda)$ and $|i|=|i'|=i$.
Similarly, we denote by $QSST(\lambda)$ the set of semi-standard marked shifted tableaux of shape $\lambda$ without the diagonal condition PST4.
Then for a given set ${ \pmb s}=(s_{ij})\in ST(\lambda,\mathbb{C})$ of variables, 
{\it Schur $Q$-multiple zeta functions} of {shape} $\lambda$ are defined to be
\begin{equation}
\label{def:Q}
    \zeta_\lambda^Q({ \pmb s})=\sum_{M\in QSST(\lambda)}\frac{1}{M^{ \pmb s}},
\end{equation}
where $M^{ \pmb s}=\displaystyle{\prod_{(i, j)\in SD(\lambda)}|m_{ij}|^{s_{ij}}}$ for $M=(m_{ij})\in QSST(\lambda)$.
For a strict partition $\lambda=(\lambda_1,\ldots,\lambda_r)$, by the definitions of $\zeta_\lambda^P$ and $\zeta_\lambda^Q$, it holds that 
\begin{equation}
\label{pq}
    \zeta_\lambda^Q({ \pmb s})=2^r\zeta_\lambda^P({ \pmb s}).
\end{equation}
As in Introduction, we define the truncated $P$- and $Q$-multiple zeta functions:

 For a fixed positive integer $N\in \mathbb{N}$, let $PSST_N(\lambda)$ and $QSST_N(\lambda)$ be the sets of all $(m_{ij})\in PSST(\lambda)$ and $QSST(\lambda)$ such that $m_{ij}\le N$ for all $i,j$.
 Then, we define
\[
    \zeta_\lambda^{P,N}({ \pmb s})=\sum_{M\in PSST_N(\lambda)}\frac{1}{M^{ \pmb s}},\text{ and }
    \zeta_\lambda^{Q,N}({ \pmb s})=\sum_{M\in QSST_N(\lambda)}\frac{1}{M^{ \pmb s}}.
\]
In this section, we prove some basic properties of the Schur $P$- and $Q$-multiple zeta functions.
{We first consider the domain of absolute convergence of the series \eqref{def:P} and \eqref{def:Q}.}
\begin{lemma}
\label{lem:convergence}
 Let
\[
 W_{\lambda}^Q
=
\left\{{\pmb s}=(s_{ij})\in ST(\lambda,\mathbb{C})\,\left|\,
\begin{array}{l}
 \text{$\Re(s_{ij})\ge 1$ for all $(i,j)\in SD(\lambda) \setminus SC(\lambda)$} \\[3pt]
 \text{$\Re(s_{ij})>1$ for all $(i,j)\in SC(\lambda)$}
\end{array}
\right.
\right\}.
\]
 Then, the series \eqref{def:P} and \eqref{def:Q} converge absolutely if ${\pmb s}\in W_\lambda^Q$.
\end{lemma} 
\begin{proof}
By \eqref{pq}, it suffices to deal with $\zeta_\lambda^Q$. Let $\lambda$ be a strict partition and $SC(\lambda)=\{(i_1,j_1),\ldots,(i_k,j_k)\}$ where $i_1<\cdots<i_k$ and $j_1>\cdots>j_k$.
Since \[\left|\sum_{\substack{M\in QSST(\lambda)\\m_{ij}\le N}}\frac{1}{M^{ \pmb s}}\right|\le \prod_{\ell=1}^k\sum_{\substack{M\in QSST(\lambda_\ell)\\m_{ij}\le N}}\prod_{(i,j)\in SD(\lambda_\ell)}\frac{1}{|m_{ij}|^{\Re(t_{ij,\ell})}},
\]
where $\lambda_\ell=(j_\ell-i_{\ell-1},j_\ell-i_{\ell-1}-1,\ldots,j_\ell-i_\ell+1)$ and $t_{ij,\ell}=s_{i+i_{\ell-1},j+i_{\ell-1}}$,
we prove that for $\lambda=(\lambda_1,\ldots,\lambda_r):=(\lambda_1,\lambda_1-1,\ldots,\lambda_1-r+1)$,
\begin{equation}
\label{partss}
    \sum_{\substack{M\in QSST(\lambda)\\ m_{ij}\le N}}\prod_{(i,j)\in SD(\lambda)}\frac{1}{|m_{ij}|^{\Re(s_{ij})}}
\end{equation}
converges absolutely in $\pmb s\in W_\lambda^Q$ as $N\rightarrow\infty$. Rearranging the order of summation, we have
\begin{align*}
    \sum_{\substack{M\in QSST(\lambda)\\ m_{ij}\le N}}\prod_{(i,j)\in SD(\lambda)}\frac{1}{|m_{ij}|^{\Re(s_{ij})}}&=\sum_{N_1=1}^N\sum_{\substack{(m_{ij})\in QSST(\lambda)\\ m_{r\lambda_r}= N_1}}\prod_{\substack{(i,j)\in SD(\lambda)\\(i,j)\neq(r,\lambda_r)}}\frac{1}{|m_{ij}|^{\Re(s_{ij})}}\frac{1}{N_1^{\Re(s_{r\lambda_r})}},
    \intertext{where $\lambda_r=\lambda_1-r+1$. By extending the region of summation and product, it holds that }
    \sum_{\substack{M\in QSST(\lambda)\\ m_{ij}\le N}}\prod_{(i,j)\in SD(\lambda)}\frac{1}{|m_{ij}|^{\Re(s_{ij})}}&\le2^{r\lambda_r}\sum_{N_1=1}^N\underset{(i,j)\ne (r,\lambda_r)}{\prod^{r}_{i=1}\prod^{\lambda_r}_{j=1}}\sum^{N_1}_{m_{ij}=1}\frac{1}{m_{ij}}\frac{1}
    {N_1^{\Re(s_{r\lambda_r})}}.
    \intertext{Since for any $\varepsilon>0$, there exists a constant $C_{\varepsilon}>1$ such that \[\sum^{N}_{m_{ij}=1}\frac{1}{m_{ij}}<\frac{C_\varepsilon}{2} N^\varepsilon,\]
    we can estimate that }
    \sum_{\substack{M\in QSST(\lambda)\\ m_{ij}\le N}}\prod_{(i,j)\in SD(\lambda)}\frac{1}{|m_{ij}|^{\Re(s_{ij})}}&\le C_\varepsilon^{r\lambda_r}\sum_{N_1=1}^N\frac{N_1^{\varepsilon r\lambda_r}}
    {N_1^{\Re(s_{r\lambda_r})}}.
\end{align*}
We can choose a sufficiently small $\varepsilon$ such that $s_{r\lambda_r}-\varepsilon r\lambda_r>1$. Thus, \eqref{partss} converges absolutely and we obtain the lemma.
\end{proof}

\begin{remark}
The condition variables ${\pmb s}$ are in $W_\lambda^Q$ is a sufficient for that the series \eqref{def:P} and \eqref{def:Q} converge absolutely.
\end{remark}
Next, we show that a Schur $Q$-multiple zeta function can be written as a linear combination of the multiple zeta (star) functions as well as the Schur multiple zeta functions.  
 Indeed, for strict partition $\lambda$ of $n$,
let $\mathcal{SF}(\lambda)$ be the set of all bijections $f:SD(\lambda)\to\{1,2,\ldots,n\}$
 satisfying the following two conditions:
\begin{itemize}
\item[(i)]
 for all $i$, $f((i,j_1))<f((i,j_2))$ if and only if $j_1<j_2$, 
\item[(ii)]
 for all $j$, $f((i_1,j))<f((i_2,j))$ if and only if $i_1<i_2$.
\end{itemize} 
For $\pmb t=(t_{ij})\in ST(\lambda,\mathbb C)$, 
 put
\[
 V(\pmb t)=
\left\{\left.
\left(t_{f^{-1}(1)},t_{f^{-1}(2)},\ldots,t_{f^{-1}(n)}\right)\in \mathbb C^{n}\,\right|\,
f\in \mathcal{SF}(\lambda)
\right\}.
\] 
We write ${\pmb w} \preceq_s \pmb t$ for ${\pmb w}=(w_1,w_2,\ldots,w_m)\in \mathbb C^m$
 if there exists $(v_1,v_2,\ldots,v_{n})\in V(\pmb t)$ satisfying the following:
 for all $1\le k\le m$, there exist $1\le h_k\le m$ and $l_k\ge 0$ such that  
\begin{itemize}
\item[(i)]
 $w_k=v_{h_k}+v_{h_k+1}+\cdots +v_{h_k+l_k}$,
 \item[(ii)]
 there are no $(i_1,i_2;j_1,j_2)$ with $i_1<i_2$ and $j_1<j_2$ such that $\{t_{i_1j_1},t_{i_1j_2},t_{i_2j_2}\}\subset\{v_{h_k},v_{h_k+1},\ldots ,v_{h_k+l_k}\}$, and
\item[(iii)]
 $\bigsqcup^{m}_{k=1}\{h_k,h_k+1,\ldots,h_k+l_k\}=\{1,2,\ldots,n\}$.
\end{itemize}
 Then, by the definition, we have 
\begin{equation}
\label{for:SchurtoMZV}
 \zeta_\lambda^Q({\pmb s})
=\sum_{{\pmb t} \,\preceq_s\, {\pmb s}}2^{m({\pmb t})}\zeta({\pmb t}),
\end{equation} 
where $m(\pmb t)$ is positive integer depend on $\pmb t$. More precisely, $m(\pmb t)$ is depended on the way to change the comma $,$ to the plus $+$.
Moreover, by an inclusion-exclusion principle,
 one can also obtain its ``dual'' expression 
\begin{equation}
\label{for:SchurtoMZSV}
 \zeta_\lambda^Q({\pmb s})
=\sum_{{\pmb t} \,\preceq_s\, {\pmb s}}(-1)^{n-{\rm dep}({\pmb t})}2^{m({\pmb t})}\zeta^{\star}({\pmb t}).
\end{equation} 
Combining (\ref{for:SchurtoMZV}) and (\ref{for:SchurtoMZSV}) with identity (\ref{pq}), we can decompose a Schur $P$-multiple zeta function into linear combination of the multiple zeta (star) functions.

\begin{exam}

 For ${\pmb s}=(s_{ij})\in ST((3,1),\mathbb{C})$, we have 
\begin{align*}
 V({\pmb s})
&=\{
(s_{11},s_{12},s_{13},s_{22}),
(s_{11},s_{12},s_{22},s_{13})
\}.
\end{align*}
 One can confirm that ${\pmb t} \preceq_s {\pmb s}$ if and only if ${\pmb t}$ is one of the following:
\begin{align*}
& (s_{11},s_{12},s_{13},s_{22}),(s_{11}+s_{12},s_{13},s_{22}),
 (s_{11},s_{12}+s_{13},s_{22}),(s_{11},s_{12},s_{13}+s_{22}),\\
& (s_{11}+s_{12}+s_{13},s_{22}),(s_{11}+s_{12},s_{13}+s_{22}),(s_{11},s_{12}+s_{13}+s_{22}),\\
& (s_{11},s_{12},s_{22},s_{13}), (s_{11}+s_{12},s_{22},s_{13}),(s_{11},s_{12}+s_{22},s_{13}).
\end{align*} 
 This shows that when ${\pmb s}\in W_{(3,1)}^Q$
\begin{align*}
\ytableausetup{boxsize=normal,aligntableaux=center}
 \ytableaushort{{s_{11}}{s_{12}}{s_{13}},{\none}{s_{22}}}\, 
=&16\zeta(s_{11},s_{12},s_{13},s_{22})+8\zeta(s_{11}+s_{12},s_{13},s_{22})+8\zeta(s_{11},s_{12}+s_{13},s_{22})\\
&+16\zeta(s_{11},s_{12},s_{13}+s_{22})+4\zeta(s_{11}+s_{12}+s_{13},s_{22})\\
&+8\zeta(s_{11}+s_{12},s_{13}+s_{22})+4\zeta(s_{11},s_{12}+s_{13}+s_{22})\\
&+16\zeta(s_{11},s_{12},s_{22},s_{13})+8\zeta(s_{11}+s_{12},s_{22},s_{13})+8\zeta(s_{11},s_{12}+s_{22},s_{13})\\
=&16\zetas(s_{11},s_{12},s_{13},s_{22})-8\zetas(s_{11}+s_{12},s_{13},s_{22})-8\zetas(s_{11},s_{12}+s_{13},s_{22})\\
&-16\zetas(s_{11},s_{12},s_{13}+s_{22})+4\zetas(s_{11}+s_{12}+s_{13},s_{22})\\
&+8\zetas(s_{11}+s_{12},s_{13}+s_{22})+4\zetas(s_{11},s_{12}+s_{13}+s_{22})\\
&+16\zetas(s_{11},s_{12},s_{22},s_{13})-8\zetas(s_{11}+s_{12},s_{22},s_{13})-8\zetas(s_{11},s_{12}+s_{22},s_{13}).
\end{align*} 
\end{exam}
\begin{exam}It holds that
\label{qtozeta}
\begin{align*}
    \zeta_{(r)}^Q\left(\ytableaushort{{s_{11}}{\cdots}{s_{1r}}}\right)&=\sum_{\pmb \ell}2^{{\rm dep}({\pmb \ell})}\zeta(\pmb \ell),\\
    \zeta_{(r)}^Q\left(\ytableaushort{{s_{11}}{\cdots}{s_{1r}}}\right)&=\sum_{\pmb \ell}(-1)^{r-{\rm dep({\pmb \ell})}}2^{{\rm dep}({\pmb \ell})}\zetas(\pmb \ell),
\end{align*}
where $\pmb \ell$ runs over all indices of the form
$\pmb \ell = (s_{11}\square s_{12}\square\cdots\square s_{1r})$
in which each $\square$ is filled by the comma $,$ or the plus $+$. 
\end{exam}
By \eqref{pq}, Schur $P$-multiple zeta functions also can be decomposed into a linear combination of multiple zeta (star) functions.

We next give short observation for a relation between Schur $Q$-multiple zeta values and Two-One formula conjectured by Ohno and Zudilin \cite{oz}, proved by Zhao \cite{z16}. 
\begin{theorem}[Two-One formula \cite{oz,z16}]
\label{twoone}
For a non-negative integer $k$, we denote $\mu_{2k+1}=(1,\{2\}^k)$. Then
for any admissible index $\pmb k= (k_1,\ldots,k_r)$ with odd entries $k_1,\ldots,k_r$, the following identities are valid:
\begin{align*}
\zetas(\mu_{k_1},\ldots,\mu_{k_r}) &=\sum_{\pmb \ell\preceq\pmb k}2^{{\rm dep}({\pmb \ell})}\zeta(\pmb \ell),\\
    &=\sum_{\pmb \ell\preceq\pmb k}(-1)^{r-{\rm dep({\pmb \ell})}}2^{{\rm dep}({\pmb \ell})}\zetas(\pmb \ell),
\end{align*}
where the sum $\displaystyle{\sum_{\pmb \ell\preceq\pmb k}}$ runs over all indices of the form
$\pmb \ell = (k_{1}\square k_{2}\square\cdots\square k_{r})$
in which each $\square$ is filled by the comma $,$ or the plus $+$. 
\end{theorem}
Combining Theorem \ref{twoone} with Lemma \ref{qtozeta}, we have the following theorem.
\begin{theorem}
\label{Qstar}
For $r$-tuple $(k_1,\ldots,k_r)$ of positive odd integers with $k_r\ge3$,
\begin{align*}
    \zeta_{(r)}^Q\left(\ytableaushort{{k_1}{k_2}{\cdots}{k_r}}\right)&=\zetas(\mu_{k_1},\ldots,\mu_{k_r})\\
    &=\zetas(1,\{2\}^{\frac{k_1-1}{2}},\ldots,1,\{2\}^{\frac{k_r-1}{2}}).
\end{align*}
\end{theorem}
This theorem gives a non-trivial identity between a single Schur $Q$-multiple zeta value and multiple zeta star value.
\begin{cor}
\label{113}
For a positive integer $k\ge4$
\[\zeta_{(k-2)}^Q\left(\ytableaushort{{1}{\cdots}{1}{3}}\right)=\frac{k-1}{2}\zeta_{(1)}^Q\left(\ytableaushort{{k}}\right).\]
\end{cor}

\section{Pfaffian expression of the Schur $Q$-multiple zeta functions}
\label{pfaffiansection}
We defined the Schur $P$- and $Q$-multiple zeta functions in parallel, but there are some properties only the Schur $Q$-multiple zeta function satisfies as in the original Schur $P$- and $Q$-polynomials.
In fact, it is known that the original Schur $Q$-polynomial has a pfaffian expression \cite{ma}. 
In this section, we give a pfaffian expression of the Schur $Q$-multiple zeta function by following the Stembridge way in \cite{st}.
We first recall the definition of pfaffian by comparing with the determinant. 
Let $\mathfrak S_n$ be the symmetric group of degree $n$. Then, for a given square matrix $A=(a_{ij})_{1\leq i,j\leq n}$, the determinant $\det (A)$ is defined by
\[\det (A)=\sum_{\sigma\in \mathfrak S_{n}}
 {\rm sgn}(\sigma)\prod_{i=1}^n a_{i,\sigma(i)},\]
where ${\rm sgn}(\sigma)$ is the signature of $\sigma$.

For the definition of the pfaffian, we define a set $\mathfrak F_{2n}$, a subset of the symmetric group $\mathfrak S_{2n}$ of even degree,
\[\mathfrak F_{2n}=\left\{\pi\in\mathfrak S_{2n}\left| \begin{array}{c}
\pi(1) < \pi(3) < \cdots < \pi(2n-1),\\
\pi(1) < \pi(2), \pi(3) < \pi(4), \ldots,\pi(2n-1) < \pi(2n)
\end{array}\right.\right\}.\]
For an ordered $2n$-tuple $\pmb v=(v_1,\ldots,v_{2n})$ of vertices, we say that a set of edges $\{\pi=((v_i,v_j),\ldots,(v_k,v_l))\}$ on $\pmb v$ is a $1$-factor if each $v_i$ is incident with exactly one edge. 
\begin{exam}
The followings are $1$-factors of $\{1,2,3,4\}$.
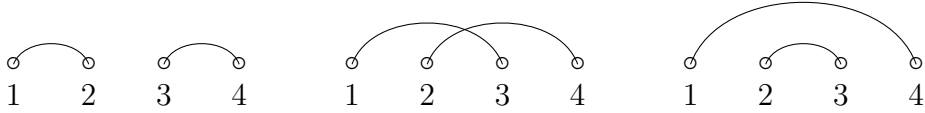
\begin{center}
 \begin{tikzpicture} 
   \node at (1,1) {$\circ$};
    \node at (2,1) {$\circ$};
    \node at (3,1) {$\circ$};
    \node at (4,1) {$\circ$};
    \node at (13,1) {$\circ$};
    \node at (5.5,1) {$\circ$};
    \node at (6.5,1) {$\circ$};
    \node at (8.5,1) {$\circ$};
    \node at (7.5,1) {$\circ$};
    \node at (10,1) {$\circ$};
    \node at (11,1) {$\circ$};
    \node at (12,1) {$\circ$};
       \node at (1,0.6) {$1$};
    \node at (2,0.6) {$2$};
    \node at (3,0.6) {$3$};
    \node at (4,0.6) {$4$};
    \node at (13,0.6) {$4$};
    \node at (5.5,0.6) {$1$};
    \node at (6.5,0.6) {$2$};
    \node at (8.5,0.6) {$4$};
    \node at (7.5,0.6) {$3$};
    \node at (10,0.6) {$1$};
    \node at (11,0.6) {$2$};
    \node at (12,0.6) {$3$};
    \draw (1,1) to [out=70,in=110] (2,1);
 \draw (3,1) to [out=70,in=110] (4,1);
 \draw (5.5,1) to [out=70,in=110] (7.5,1);
 \draw (6.5,1) to [out=70,in=110] (8.5,1);
 \draw (10,1) to [out=70,in=110] (13,1);
 \draw (11,1) to [out=70,in=110] (12,1);
 \end{tikzpicture}
\end{center}
\end{exam}
By convention, we always list the edges of a $1$-factor $\pi$ in the form $(v_i,v_j)$ with $i<j$.
It is known that one can compose a bijection from $\mathfrak F_{2n}$ to the set of $1$-factors by
$\pi\mapsto \{(v_{\pi(1)},v_{\pi(2)}),\ldots,(v_{\pi(2n-1)},v_{\pi(2n)})\}$,
and \[|\mathfrak F_{2n}|=\frac{(2n)!}{2^nn!}.\]
Then, for a given $2n \times 2n$ upper triangular or anti-symmetric matrix $A=(a_{ij})_{1\leq i,j\leq2n}$,
 the pfaffian ${\rm pf} (A)$ of $A$ is defined by
\[
{\rm pf} (A) = \sum_{\pi\in \mathfrak F_{2n}}
 {\rm sgn}(\pi)\prod_{i=1}^n a_{\pi(2i-1),\pi(2i)}.
\]

Let $D=(V,E)$ be a directed acyclic graph with vertices $V$ and edges $E$, an assignment of a
direction to each edge with no directed cycles. 
We denote by $u\rightarrow v$ an edge directed from $u$ to $v$.
For any pair of vertices $u,v$, we denote by $\mathscr P(u, v)$ denote the set of directed
$D$-paths from $u$ to $v$ on $D$. If $u=u$, then $\mathscr P(u, u)$ is a set of a single path of length zero.

Let $I$ and $J$ be ordered sets of vertices of $D$. Then $I$ is
said to be $D$-compatible with $J$ if, whenever $u<u'$ in $I$ and $v>v'$ in $J$, every
path $P\in \mathscr P(u, v)$ intersects every path $Q \in \mathscr P(u', v')$.

For any vertex $u \in V$ and subset $I\subset V$, let $\mathscr P(u; I)$ denote the set of
directed paths from $u$ to any $v\in I$, and let 
\[Q_I(u) =  \sum_{P\in \mathscr P(u; I)} w(P),\]
where $w$ is a particular weight function.

For any $r$-tuple $\pmb u = (u_1,\ldots,u_r)$ of vertices,
let $\mathscr P_0(\pmb u;I)$ be the set of non-intersecting $r$-tuples of paths $P_i\in \mathscr P_{0}(u_i;I)$.
Then we define \[Q_I(u_1,\ldots,u_r) =  \sum_{(P_1,\ldots,P_r)\in \mathscr P(\pmb u; I)} w(P_1,\ldots,P_r).\]

\begin{theorem}[{\cite[Theorem 3.1]{st}}]
\label{lem31}
Let $\pmb u = (u_1,\ldots,u_r)$ be an $r$-tuple of vertices in a directed acyclic graph $D$, and assume that $r$ is even. If $I\subset V$ is a totally ordered subset of the vertices such that $u$ is $D$-compatible with $I$, then
\[Q_I(u_1.\ldots,u_r)={\rm pf}(Q_I(u_i,u_j))_{1\le i<j\le r}.\]
\end{theorem}
\begin{remark}[\cite{st}]
  In case $r$ is odd, we may adjoin a phantom vertex $u_{r+1}$ to $V$,
with no incident edges, and include $u_{r+1}$ in $I$. We order all other vertices of $I$ before $u_{r+1}$ and replace $r$ by $r+1$.
\end{remark}
Stembridge composed a directed graph $D$ corresponding to the Schur $Q$-functions \cite{st}. 
Moreover, Stembridge applied Theorem \ref{lem31} to obtain the following pfaffian expression of the Schur $Q$-polynomial.
\begin{theorem}[{\cite[Theorem 6.1]{st}}]
Let $\lambda=(\lambda_1,\ldots,\lambda_r)$ be a strict partition of even length.
Then 
\[Q_\lambda={\rm pf}(Q_{(\lambda_i,\lambda_j)})_{1\le i<j\le r}.\]
\end{theorem}
Following the Stembridge way, we compose a directed graph $D$ corresponding to the Schur $Q$-multiple zeta functions.
We begin with the vertex
set of non-negative integers, and direct an edge $u\rightarrow v$ whenever $u-v = (1, 0), (0, 1)$, or $(1, 1)$. Subsequently, we delete the edges $u\rightarrow v$ involving points whose first
coordinates are both zero, as well as those whose second coordinates are
both zero. 
Finally, we split each of the vertices $(0,j)$ with $j> 1$ into two vertices, say $(0, j)$ and $(0, j+1)'$, so that the edge $(1,j + 1)\rightarrow (0,j)$ is redirected
to $(0, j+1)'$, while the edge $(1, j)\rightarrow (0, j)$ remains untouched. 
Fix a positive integer $N$, and let $\pmb u = (u_1,\ldots,u_r)$ be the $r$-tuple of vertices with
$u_i = (\lambda_i, N)$. Without loss of generality, we may assume that $r$ is even (if $r$
is odd, set $\lambda_{r+1}= 0$ and $u_{r+1} = (0, N+1)'$, and replace $r$
by $r+ 1$). 

Let $I_N=\{(0,0),(0,1),(0,2)',(0,2),\ldots,(0,N)',(0,N),(0,N+1)'\}$.
Then Stembridge showed that the element in $QSST_N(\lambda)$ can be identified with the non-intersecting paths in $\mathscr P_{0}(\pmb u;I_N)$, and $\pmb u$ is $D$-compatible with $I_N$. 

Let $v_i(P)=(v_{ij}(P))_{j=0}$ be the sequence of vertices representing the path $P\in \mathscr P_{0}(u_i;I_N)$. Successively, let $v_i^w(P)=(v_{ij}^w(P))_{j=0}$ be the sub-sequence of $v_i(P)$ with $v_{ij}(P)-v_{i(j+1)}(P)= (1,0)$ and $(1,1)$, let $v_i^1(P)=(v_{ij}^1(P))_{j=0}$ be the sub-sequence of $v_i(P)$ with $v_{ij}(P)-v_{i(j+1)}(P)= (0,1)$.
For $\pmb s\in ST(\lambda,\mathbb C)$, we assign the weight $w(v_{ij}^w(P))=v_2^{-s_{i(\lambda_i-j)}}$ to $v_{ij}^w(P)=(v_1,v_2)$ and we assign the weight $w(v_{ij}^1(P))=1$ to $v_{ij}^1(P)$. 
Then, we define \[w(P)=\prod_{v_{ij}(P)}w(v_{ij}(P)),\]
and for $(P_1,\ldots,P_r)\in \mathscr P(\pmb u;I_N)$,
\[w(P_1,\ldots,P_r)=\prod_{i=1}^rw(P_i).\]
\begin{exam}Let $\lambda=(6,5,3,1)$. Then, the Figure \ref{figure} is a $4$-tuple of paths $(P_1,P_2,P_3,P_4)\in \mathscr P(\{u_1,u_2\};I_N)\oplus\mathscr P(\{u_3,u_4\};I_N)$.
\label{35}
\begin{figure}[ht]
\begin{center}
 \begin{tikzpicture} 
     \node at (7,6.5) {$u_1$};
     \node at (6,6.5) {$u_2$};
     \node at (4,6.5) {$u_3$};
     \node at (2,6.5) {$u_4$};
     \node at (0.7,0.7) {$O$};
     \node at (7,0.7) {$6$};
     \node at (6,0.7) {$5$};
     \node at (4,0.7) {$3$};
     \node at (2,0.7) {$1$};
     \node at (3,0.7) {$2$};
     \node at (5,0.7) {$4$};
     \node at (0.3,2.5) {$(0,2)'$};
   \node at (0.3,3.5) {$(0,3)'$};
   \node at (0.3,5) {$(0,4)$};
   \node at (1,1) {$\bullet$};
   \node at (1,2) {$\bullet$};
   \node at (1,3) {$\bullet$};
   \node at (1,4) {$\bullet$};    
   \node at (2,1) {$\bullet$};
   \node at (2,2) {$\bullet$};
   \node at (2,3) {$\bullet$};  
   \node at (2,4) {$\bullet$};  
   \node at (3,1) {$\bullet$};
   \node at (3,2) {$\bullet$};
   \node at (3,3) {$\bullet$};  
  \node at (3,4) {$\bullet$};  
   \node at (4,1) {$\bullet$};
   \node at (4,2) {$\bullet$};
   \node at (4,3) {$\bullet$};  
   \node at (4,4) {$\bullet$};  
   \node at (5,1) {$\bullet$};
   \node at (5,2) {$\bullet$};
   \node at (5,3) {$\bullet$};  
   \node at (5,4) {$\bullet$};  
   \node at (6,1) {$\bullet$};
   \node at (6,2) {$\bullet$};
   \node at (6,3) {$\bullet$};  
   \node at (6,4) {$\bullet$};  
    \node at (7,1) {$\bullet$};
   \node at (7,2) {$\bullet$};
   \node at (7,3) {$\bullet$};  
   \node at (7,4) {$\bullet$}; 
   \node at (1,5) {$\bullet$};
   \node at (2,5) {$\bullet$};  
   \node at (3,5) {$\bullet$}; 
   \node at (4,5) {$\bullet$};
   \node at (5,5) {$\bullet$};
   \node at (6,5) {$\bullet$};  
   \node at (7,5) {$\bullet$}; 
   \node at (1,6) {$\bullet$};
   \node at (2,6) {$\bullet$};  
   \node at (3,6) {$\bullet$}; 
   \node at (4,6) {$\bullet$};
   \node at (5,6) {$\bullet$};
   \node at (6,6) {$\bullet$};  
   \node at (7,6) {$\bullet$}; 
   \draw[->,thick] (1,0.5)--(1,6.5);
   \draw[->,thick] (0.5,1)--(7.5,1);
\draw (6,6) -- (5,5) -- (4,4) -- (3,4) -- (2,4)--(1,3.5);
\draw[loosely dashdotted] (4,6) -- (3,5) -- (3,4) -- (3,3) -- (2,3)--(1,2.5); 
  \draw[dotted] (7,6) -- (7,5) -- (4,2) -- (1,2);
   \draw[dashed] (2,6) -- (2,5) -- (1,4) ;
 \end{tikzpicture}
\end{center}
\caption{$(P_1,P_2,P_3,P_4)$ in Example \ref{35}}
 \label{figure}
\end{figure}
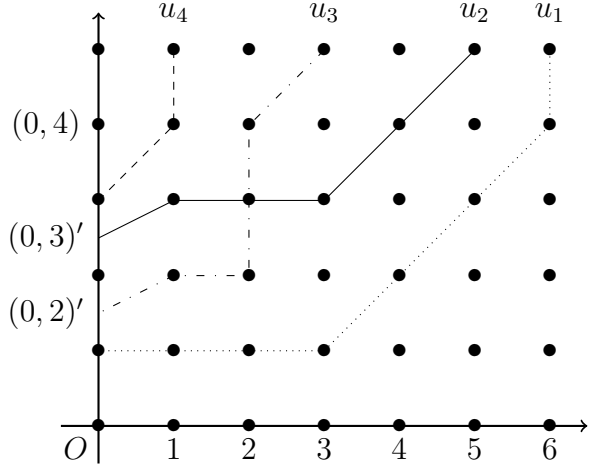

Let $(s_{ij})\in ST(\lambda,\mathbb C)$. The weight $w(P_i)$ are 
\begin{align*}
    w(P_1)&=\frac{1}{1^{s_{11}}1^{s_{12}}1^{s_{13}}2^{s_{14}}3^{s_{15}}4^{s_{16}}},&&w(P_2)=\frac{1}{3^{s_{22}}3^{s_{23}}3^{s_{24}}4^{s_{25}}5^{s_{26}}}\\
    w(P_3)&=\frac{1}{2^{s_{33}}3^{s_{34}}5^{s_{35}}},&&w(P_4)=\frac{1}{4^{s_{44}}}.
\end{align*}
\end{exam}
Then, due to the Stembridge composition, we find that
\[\zeta_\lambda^{Q,N}(\pmb s)=\sum_{(P_1,\ldots,P_r)\in \mathscr P_0(\pmb u;I_N)}w(P_1,\ldots,P_r).\]

For a set $X$, we define $ST^{\mathrm{diag}}(\lambda,X)=\{(t_{ij})\in ST(\lambda,X)~|~\text{$t_{ij}=t_{1k}$ if $j-i=k-1$}\}$.
\begin{theorem}[Pfaffian expression of the Schur $Q$-multiple zeta functions]
\label{pfaffian}
Let $\lambda=(\lambda_1,\ldots,\lambda_{r})$ be a strict partition into even parts with $\lambda_i\ge0$. 
Then for $\pmb s\in ST^{\rm diag}(\lambda,\mathbb C)$,
\[\zeta_{\lambda}^Q(\pmb s)={\rm pf}(M_\lambda),\]
where $M_\lambda=(a_{ij})$ is an $r\times r$ upper triangular matrix with \[a_{ij}=\zeta_{(\lambda_i,\lambda_j)}^Q(\pmb s_{(\lambda_i,\lambda_j)})\]
and 
\[\pmb s_{(\lambda_i,\lambda_j)}=\ytableaushort{{s_{ii}}{\cdots}{\cdots}{\cdots}{s_{it_i}},{\none}{s_{jj}}{\cdots}{s_{jt_j}}},\]
 where $t_i=i+\lambda_i-1$.
\end{theorem}
\begin{proof}
By the definition of pfaffian,
\begin{equation}
    \label{pffdef}
    {\rm pf}(M_\lambda)=\sum_{\pi\in\mathscr F_n}{\rm sgn}(\pi)\prod_{(i,j)\in\pi}\zeta_{(\lambda_i,\lambda_j)}^Q(\pmb s_{(\lambda_i,\lambda_j)}).
\end{equation}
It suffices to show that there exists a sign-reversing summand for each summand caused from $(P_1,\ldots,P_r)$ with at least one pair of intersecting paths.

We consider the right-most intersection point $(p,q)$ appearing in paths $(P_1,\ldots,P_r)$ for a 1-factor $\pi$.
For the sake of simplicity, we can assume that for 1-factor $\pi$ the two paths $P_1$ and $P_2$ cross at $(p,q)$ (Figure \ref{fig2}).
Then, the paths $(P_1,\ldots,P_r)$ gives \[S(\pi)={\rm sgn}(\pi)\prod_{j=1}^{t_1}a_{1j}^{-s_{1j}}\prod_{j=2}^{t_2}a_{2j}^{-s_{2j}}\prod_{i=3}^r\left(\prod_{j=i}^{t_i}a_{ij}^{-s_{ij}}\right),\]
where $a_{ij}$ is the $y$-coordinate of corresponding element of $v_{ij}^w(P_i)$.
On the other hand, we consider the $r$-tuple of paths $(\overline P_1,\overline{P}_2,P_3,\ldots,P_r)$.
Here, $\overline P_i$ follows $P_i$ until it meets the first intersection point $(p,q)$ and after that follows the other path $P_j$ to the end (Figure \ref{fig3}). 
\begin{figure}[ht]
\begin{tabular}{c}
\begin{minipage}{0.5\hsize}
\begin{center}
 \begin{tikzpicture} 
     \node at (7,6.5) {$u_1$};
     \node at (6,6.5) {$u_2$};
     \node at (4,6.5) {$u_3$};
     \node at (2,6.5) {$u_4$};
     \node at (4,0.5) {$p$};
     \node at (0.5,4) {$q$};
   \node at (1,1) {$\bullet$};
   \node at (1,2) {$\bullet$};
   \node at (1,3) {$\bullet$};
   \node at (1,4) {$\bullet$};    
   \node at (2,1) {$\bullet$};
   \node at (2,2) {$\bullet$};
   \node at (2,3) {$\bullet$};  
   \node at (2,4) {$\bullet$};  
   \node at (3,1) {$\bullet$};
   \node at (3,2) {$\bullet$};
   \node at (3,3) {$\bullet$};  
  \node at (3,4) {$\bullet$};  
   \node at (4,1) {$\bullet$};
   \node at (4,2) {$\bullet$};
   \node at (4,3) {$\bullet$};  
   \node at (4,4) {$\bullet$};  
   \node at (5,1) {$\bullet$};
   \node at (5,2) {$\bullet$};
   \node at (5,3) {$\bullet$};  
   \node at (5,4) {$\bullet$};  
   \node at (6,1) {$\bullet$};
   \node at (6,2) {$\bullet$};
   \node at (6,3) {$\bullet$};  
   \node at (6,4) {$\bullet$};  
    \node at (7,1) {$\bullet$};
   \node at (7,2) {$\bullet$};
   \node at (7,3) {$\bullet$};  
   \node at (7,4) {$\bullet$}; 
   \node at (1,5) {$\bullet$};
   \node at (2,5) {$\bullet$};  
   \node at (3,5) {$\bullet$}; 
   \node at (4,5) {$\bullet$};
   \node at (5,5) {$\bullet$};
   \node at (6,5) {$\bullet$};  
   \node at (7,5) {$\bullet$}; 
   \node at (1,6) {$\bullet$};
   \node at (2,6) {$\bullet$};  
   \node at (3,6) {$\bullet$}; 
   \node at (4,6) {$\bullet$};
   \node at (5,6) {$\bullet$};
   \node at (6,6) {$\bullet$};  
   \node at (7,6) {$\bullet$}; 
\draw (7,6) --(6,5)-- (6,4) -- (2,4)--(1,3.5);
  \draw[loosely dashdotted] (6,6) -- (3,3) -- (3,2) -- (1,2);
\draw[dotted] (4,6) -- (3,5) -- (3,4) -- (2,3) --(1,2.5); 
   \draw[dashed] (2,6) -- (2,5) -- (1,4) ;
 \end{tikzpicture}
\end{center}
\ \\[-40pt]
\caption{$(P_1,\ldots,P_r)$}
 \label{fig2}
\end{minipage}
\begin{minipage}{0.5\hsize}
\begin{center}
 \begin{tikzpicture} 
     \node at (7,6.5) {$u_1$};
     \node at (6,6.5) {$u_2$};
     \node at (4,6.5) {$u_3$};
     \node at (2,6.5) {$u_4$};
     \node at (4,0.5) {$p$};
     \node at (0.5,4) {$q$};
   \node at (1,1) {$\bullet$};
   \node at (1,2) {$\bullet$};
   \node at (1,3) {$\bullet$};
   \node at (1,4) {$\bullet$};    
   \node at (2,1) {$\bullet$};
   \node at (2,2) {$\bullet$};
   \node at (2,3) {$\bullet$};  
   \node at (2,4) {$\bullet$};  
   \node at (3,1) {$\bullet$};
   \node at (3,2) {$\bullet$};
   \node at (3,3) {$\bullet$};  
  \node at (3,4) {$\bullet$};  
   \node at (4,1) {$\bullet$};
   \node at (4,2) {$\bullet$};
   \node at (4,3) {$\bullet$};  
   \node at (4,4) {$\bullet$};  
   \node at (5,1) {$\bullet$};
   \node at (5,2) {$\bullet$};
   \node at (5,3) {$\bullet$};  
   \node at (5,4) {$\bullet$};  
   \node at (6,1) {$\bullet$};
   \node at (6,2) {$\bullet$};
   \node at (6,3) {$\bullet$};  
   \node at (6,4) {$\bullet$};  
    \node at (7,1) {$\bullet$};
   \node at (7,2) {$\bullet$};
   \node at (7,3) {$\bullet$};  
   \node at (7,4) {$\bullet$}; 
   \node at (1,5) {$\bullet$};
   \node at (2,5) {$\bullet$};  
   \node at (3,5) {$\bullet$}; 
   \node at (4,5) {$\bullet$};
   \node at (5,5) {$\bullet$};
   \node at (6,5) {$\bullet$};  
   \node at (7,5) {$\bullet$}; 
   \node at (1,6) {$\bullet$};
   \node at (2,6) {$\bullet$};  
   \node at (3,6) {$\bullet$}; 
   \node at (4,6) {$\bullet$};
   \node at (5,6) {$\bullet$};
   \node at (6,6) {$\bullet$};  
   \node at (7,6) {$\bullet$}; 
\draw (7,6) --(6,5)-- (6,4) --(4,4)-- (3,3) -- (3,2) -- (1,2);
  \draw[loosely dashdotted] (6,6) --(4,4)-- (2,4)--(1,3.5);
\draw[dotted] (4,6) -- (3,5) -- (3,4) -- (2,3) --(1,2.5); 
   \draw[dashed] (2,6) -- (2,5) -- (1,4) ;
 \end{tikzpicture}
\end{center}
\ \\[-40pt]
\caption{$(\overline P_1,\overline{P}_2,P_3,\ldots,P_r)$}
  \label{fig3}
\end{minipage}
\end{tabular}
\end{figure}

Let $\overline\pi$ be the 1-factor obtained by interchanging $1$ and $2$.
We need to check that for each 1-factor $(i,j)\in \overline{\pi}$, the paths $P_i$ and $P_j$ do
not intersect. It suffices to deal with the cases involving the modified paths $\overline{P}_i$ and $\overline{P}_j$. The definition of $v$ implies that there are no points of intersection other than  $v$ on the right-hand side of $v$. Hence, the path $P_k$ will intersect $P_1$ (resp. $P_2$) if and only if $P_k$ intersects $\overline P_2$ (resp. $\overline{P}_1$).  Thus, we confirm that $\overline{\pi}$ appears in (\ref{pffdef}) and
the paths $(\overline{P}_1,\overline{P}_2,\ldots,P_r)$ gives \[S(\overline\pi)={\rm sgn}(\overline\pi)\prod_{j=2}^{p+1}a_{2j}^{-s_{2j}}\prod_{j=p+1}^{t_1}a_{1j}^{-s_{1j}}\prod_{j=1}^{p}a_{1j}^{-s_{1j}}\prod_{j=p+2}^{t_2}a_{2j}^{-s_{2j}}\prod_{i=3}^r\left(\prod_{j=i}^{t_i}a_{ij}^{-s_{ij}}\right).\]

As ${\rm sgn}(\pi){\rm sgn}(\overline\pi)=-1$ and $a_{1j}=a_{2(j+1)}$, one can confirm that \[S(\pi)+S(\overline{\pi})=0,\]
and this proves the assertion.
\end{proof}
\begin{exam}Let $\lambda=(3,2,1,0)$. Then, if $(a_{j-i})=(s_{ij})\in ST^{\rm diag}(\lambda,\mathbb C)$,
\begin{align*}
    \zeta_\lambda^Q(\pmb s)=&{\rm pf}\begin{pmatrix}
0&\zeta_{(3,2)}^Q\left(\ytableaushort{{a_{0}}{a_1}{a_2},{\none}{a_{0}}{a_1}}\right)&\zeta_{(3,1)}^Q\left(\ytableaushort{{a_{0}}{a_{1}}{a_{2}},{\none}{a_{0}}}\right)&\zeta_{(3)}^Q\left(\ytableaushort{{a_0}{a_1}{a_2}}\right)\\
0&0&\zeta_{(2,1)}^Q\left(\ytableaushort{{a_0}{a_1},{\none}{a_0}}\right)&\zeta_{(2)}^Q\left(\ytableaushort{{a_0}{a_1}}\right)\\
0&0&0&\zeta_{(1)}^Q\left(\ytableaushort{{a_0}}\right)\\
0&0&0&0
\end{pmatrix}\\
=&\zeta_{(3,2)}^Q\left(\ytableaushort{{a_{0}}{a_1}{a_2},{\none}{a_{0}}{a_1}}\right)\zeta_{(1)}^Q\left(\ytableaushort{{a_0}}\right)-\zeta_{(3,1)}^Q\left(\ytableaushort{{a_{0}}{a_{1}}{a_{2}},{\none}{a_{0}}}\right)\zeta_{(2)}^Q\left(\ytableaushort{{a_0}{a_1}}\right)\\
&+\zeta_{(3)}^Q\left(\ytableaushort{{a_0}{a_1}{a_2}}\right)\zeta_{(2,1)}^Q\left(\ytableaushort{{a_0}{a_1},{\none}{a_0}}\right).
\end{align*}
\end{exam}
As in \cite{nt}, we can consider the extension of Theorem \ref{pfaffian}. 
{In preparation, we define 
\[\sum_{{\rm diag}}=\sum_{\substack{\sigma_j\in S_j\\j\in \mathbb Z}}\prod_{i\in \mathbb Z}\sigma_i\]
for $S_j$ being the set of permutations of the elements of $I(j)=\{(k,l)\in SD(\lambda)~|~l-k=j\}$. 
The sum $\displaystyle \sum_{{\rm diag}}$ means the sum taken over all permutations of all elements on each diagonal $I(j)$. We have to note that since the number of boxes in a fixed shifted tableau is finite, the product and the sum are finite. Actually, for $\lambda=(\lambda_1,\ldots,\lambda_r)$ we find that
\[\sum_{\substack{\sigma_j\in S_j\\j\in \mathbb Z}}\prod_{i\in \mathbb Z}\sigma_i=\sum_{\substack{\sigma_j\in S_j\\j\in \mathbb Z\cap[0,\lambda_1]}}\prod_{i=0}^{\lambda_1}\sigma_i.\]
}
Also, we define a set $W_{\lambda,H}^Q$ by
\[
 W_{\lambda,H}^Q
=
\left\{{\pmb s}=(s_{ij})\in ST(\lambda,\mathbb{C})\,\left|\,
\begin{array}{l}
 \text{$\Re(s_{ij})\ge 1$ for all $(i,j)\in SD(\lambda) \setminus H(\lambda)$} \\[3pt]
 \text{$\Re(s_{ij})>1$ for all $(i,j)\in H(\lambda)$}
\end{array}
\right.
\right\},
\]
where $H(\lambda)=\{(i,j)\in SD(\lambda)~|~i-j\in\{i-\lambda_i~|~1\le i\le r\} \}$.
Following the proof of Theorem \ref{pfaffian} and \cite[Lemma 3.1]{nt}, one can prove the following theorem.
\begin{theorem}
For any strict partition $\lambda=(\lambda_1,\ldots,\lambda_r)$ and {$\pmb s\in W_{\lambda,H}^Q$}, we have
\label{diag}
\[\sum_{{\rm diag}}\zeta_\lambda^Q(\pmb{s})=\sum_{{\rm diag}}{\rm pf}(M_\lambda),\]
where $M_\lambda$ is defined as in Theorem \ref{pfaffian}.
\end{theorem}
\section{Generalization to skew type}
\label{skewsection}
For strict partitions $\lambda,\mu$, we write $\mu\le \lambda$ if $SD(\mu)\subset SD(\lambda)$. For $\mu\le \lambda$, the {\it skew shifted diagram} of $\lambda/\mu$ is defined as $SD(\lambda/\mu) = SD(\lambda)\setminus SD(\mu)$.
 We use the same notations $T(\lambda/\mu,X),T^{\mathrm{diag}}(\lambda/\mu,X)$ for a set $X$,
and $PSST(\lambda/\mu)$ as in the previous sections.

 Let ${\pmb s}=(s_{ij})\in ST(\lambda/\mu,\mathbb{C})$.
 We define skew Schur $P$- and skew $Q$-multiple zeta functions associated with $\lambda/\mu$ by
\begin{equation}
\label{def:skewSPMZ}
 \zeta_{\lambda/\mu}^P({\pmb s})
=\sum_{M\in PSST(\lambda/\mu)}\frac{1}{M^{\pmb s}},
\end{equation}
and
\begin{equation}
\label{def:skewSQMZ}
 \zeta_{\lambda/\mu}^Q({\pmb s})
=\sum_{M\in QSST(\lambda/\mu)}\frac{1}{M^{\pmb s}}.
\end{equation}
Let $D$ be the directed graph defined above and \[I_N=\{(0,0),(0,1),(0,2)',(0,2),\ldots,(0,N)',(0,N),(0,N+1)'\}\] for a fixed positive integer $N$. We define two sequences of vertices $\pmb u=(u_1,\ldots,u_r)$ and $\pmb v=(v_1,\ldots,v_s)$ by $u_i=(\lambda_i,N)$ and $v_i=(\mu_i,0)$.
We define $\pmb v\oplus I_N$ by the union of $\pmb v$ and $I_m$, ordered so that each $v_i$ precedes each $v\in I_N$.
Then the shifted Young tableaux of shape $\lambda/\mu$ with max entry $N$ can be identified with the non-intersecting paths in $\mathscr P_{0}(\pmb u,\pmb v\oplus I_N)$, and $\pmb u$ is $D$-compatible with $\pmb v\oplus I_N$. 

Let $v_i(P)=(v_{ij}(P))_{j=0}$ be the sequence of vertices representing the path $P\in \mathscr P_{0}(u_i;\pmb v\oplus I_N)$. Successively, let $v_i^w(P)=(v_{ij}^w(P))_{j=0}$ be the sub-sequence of $v_i(P)$ with $v_{ij}(P)-v_{i(j+1)}(P)= (1,0)$ and $(1,1)$, let $v_i^1(P)=(v_{ij}^1(P))_{j=0}$ be the sub-sequence of $v_i(P)$ with $v_{ij}(P)-v_{i(j+1)}(P)= (0,1)$.
For $\pmb s\in ST(\lambda,\mathbb C)$, we assign the weight $w(v_{ij}^w(P))=v_2^{-s_{i(\lambda_i-j)}}$ to $v_{ij}^w(P)=(v_1,v_2)$ and we assign the weight $w(v_{ij}^1(P))=1$ to $v_{ij}^1(P)$. 
Then, we define \[w(P)=\prod_{v_{ij}(P)}w(v_{ij}(P)),\]
and for $(P_1,\ldots,P_r)\in \mathscr P(\pmb u;\pmb v\oplus I_N)$,
\[w(P_1,\ldots,P_r)=\prod_{i=1}^rw(P_i).\]
\begin{exam}
\label{43}
Let $\lambda=(6,5,3,1)$ and $\mu=(3,1)$. Then Figure \ref{skewfigure} is a $4$-tuple of non-intersecting paths $(P_1,P_2,P_3,P_4)\in\mathscr P_{0}(\pmb u,\pmb v\oplus I_5)$.

\begin{figure}[ht]
\begin{center}
 \begin{tikzpicture} 
     \node at (7,6.5) {$u_1$};
     \node at (6,6.5) {$u_2$};
     \node at (4,6.5) {$u_3$};
     \node at (2,6.5) {$u_4$};
      \node at (4,0.5) {$v_1$};
     \node at (2,0.5) {$v_2$};
     \node at (3,0.5) {$2$};
     \node at (5,0.5) {$4$};
     \node at (6,0.5) {$5$};
     \node at (7,0.5) {$6$};
     \node at (0.3,2.5) {$(0,2)'$};
   \node at (0.3,4) {$(0,3)$};
   \node at (0.3,5) {$(0,4)$};
   \node at (0.7,0.7) {$O$};
   \node at (1,1) {$\bullet$};
   \node at (1,2) {$\bullet$};
   \node at (1,3) {$\bullet$};
   \node at (1,4) {$\bullet$};    
   \node at (2,1) {$\bullet$};
   \node at (2,2) {$\bullet$};
   \node at (2,3) {$\bullet$};  
   \node at (2,4) {$\bullet$};  
   \node at (3,1) {$\bullet$};
   \node at (3,2) {$\bullet$};
   \node at (3,3) {$\bullet$};  
  \node at (3,4) {$\bullet$};  
   \node at (4,1) {$\bullet$};
   \node at (4,2) {$\bullet$};
   \node at (4,3) {$\bullet$};  
   \node at (4,4) {$\bullet$};  
   \node at (5,1) {$\bullet$};
   \node at (5,2) {$\bullet$};
   \node at (5,3) {$\bullet$};  
   \node at (5,4) {$\bullet$};  
   \node at (6,1) {$\bullet$};
   \node at (6,2) {$\bullet$};
   \node at (6,3) {$\bullet$};  
   \node at (6,4) {$\bullet$};  
    \node at (7,1) {$\bullet$};
   \node at (7,2) {$\bullet$};
   \node at (7,3) {$\bullet$};  
   \node at (7,4) {$\bullet$}; 
   \node at (1,5) {$\bullet$};
   \node at (2,5) {$\bullet$};  
   \node at (3,5) {$\bullet$}; 
   \node at (4,5) {$\bullet$};
   \node at (5,5) {$\bullet$};
   \node at (6,5) {$\bullet$};  
   \node at (7,5) {$\bullet$}; 
   \node at (1,6) {$\bullet$};
   \node at (2,6) {$\bullet$};  
   \node at (3,6) {$\bullet$}; 
   \node at (4,6) {$\bullet$};
   \node at (5,6) {$\bullet$};
   \node at (6,6) {$\bullet$};  
   \node at (7,6) {$\bullet$}; 
   \draw[->,thick] (1,0.5)--(1,6.5);
   \draw[->,thick] (0.5,1)--(7.5,1);
\draw (6,6) -- (5,5) -- (4,4) -- (4,3)--(3,2)--(2,2) --(2,1);
\draw[loosely dashdotted] (4,6) -- (3,5) --(3,4)-- (2,3)--(1,2.5); 
  \draw[dotted] (7,6) -- (7,5) -- (5,3) -- (5,2)--(4,2)--(4,1);
   \draw[dashed] (2,6) -- (2,5) -- (1,4) ;
 \end{tikzpicture}
\end{center}
\caption{$(P_1,P_2,P_3,P_4)$ in Example \ref{43}}
 \label{skewfigure}
\end{figure}
Let $(s_{ij})\in ST(\lambda/\mu,\mathbb C)$. The weight $w(P_i)$ are 
\begin{align*}
    w(P_1)&=\frac{1}{1^{s_{14}}3^{s_{15}}4^{s_{16}}},&&w(P_2)=\frac{1}{1^{s_{23}}2^{s_{24}}4^{s_{25}}5^{s_{26}}},\\
    w(P_3)&=\frac{1}{2^{s_{33}}3^{s_{34}}5^{s_{35}}},&&w(P_4)=\frac{1}{4^{s_{44}}}.
\end{align*}
\end{exam}

Then, we find that
\[\zeta_{\lambda/\mu}^{Q,N}(\pmb s)=\sum_{(P_1,\ldots,P_r)\in \mathscr P_0(\pmb u;\pmb v\oplus I_N)}w(P_1,\ldots,P_r).\]
Moreover, we can apply Stembridge result {\cite[Theorem 3.2]{st}}, which generalizes Theorem \ref{lem31}, and obtain the following pfaffian expression of the skew Schur $Q$-multiple zeta functions.
\begin{theorem}[Pfaffian expression of the skew Schur $Q$-multiple zeta functions]
\label{skewpfaffian}
Let $\lambda=(\lambda_1,\ldots,\lambda_{r})$, $\mu=(\mu_1,\ldots,\mu_{s})$ be strict partitions into with $\lambda_i\ge0$ and $2|r+s$. 
Then for $\pmb s\in ST^{\rm diag}(\lambda/\mu,\mathbb C)$,
\[\zeta_{\lambda/\mu}^Q(\pmb s)={\rm pf}\begin{pmatrix}
M_\lambda&H_{\lambda,\mu}\\
0&0
\end{pmatrix}
,\]
where $M_\lambda=(a_{ij})$ is an $r\times r$ upper triangular matrix with \[a_{ij}=\zeta_{(\lambda_i,\lambda_j)}^Q(\pmb s_{(\lambda_i,\lambda_j)}),\]
\[\pmb s_{(\lambda_i,\lambda_j)}=\ytableaushort{{s_{ii}}{\cdots}{\cdots}{\cdots}{s_{it_i}},{\none}{s_{jj}}{\cdots}{s_{jt_j}}},\]
 where $t_i=i+\lambda_i-1$ and $H_\lambda=(b_{ij})$ is an $r\times s$ matrix with \[b_{ij}=\zeta_{(\lambda_i-\mu_s-j+1)}^Q(s_{i(i+j+\mu_s-1)},\ldots,s_{it_i}).\]
\end{theorem}
\begin{exam}Let $\lambda=(3,2,1)$ and $\mu=(2)$. Then, if $(a_{j-i})=(s_{ij})\in ST^{\rm diag}(\lambda/\mu,\mathbb C)$
\begin{align*}
    \zeta_\lambda^Q(\pmb s)=&{\rm pf}\begin{pmatrix}
0&\zeta_{(3,2)}^Q\left(\ytableaushort{{a_{0}}{a_1}{a_2},{\none}{a_{0}}{a_1}}\right)&\zeta_{(3,1)}^Q\left(\ytableaushort{{a_{0}}{a_{1}}{a_{2}},{\none}{a_{0}}}\right)&\zeta_{(1)}^Q\left(\ytableaushort{{a_2}}\right)\\
0&0&\zeta_{(2,1)}^Q\left(\ytableaushort{{a_0}{a_1},{\none}{a_0}}\right)&1\\
0&0&0&0\\
0&0&0&0
\end{pmatrix}\\
=&-\zeta_{(3,1)}^Q\left(\ytableaushort{{a_{0}}{a_{1}}{a_{2}},{\none}{a_{0}}}\right)+\zeta_{(1)}^Q\left(\ytableaushort{{a_2}}\right)\zeta_{(2,1)}^Q\left(\ytableaushort{{a_0}{a_1},{\none}{a_0}}\right).
\end{align*}
\end{exam}
\section{Outside decomposition}
\label{outsidesec}
Hamel and Goulden proved a general determinant formula which expressed a Schur function as a determinant of skew Schur functions whose shapes are {\it strips} (\cite{hg}). 
Subsequently, Hamel proved expressions of Schur $Q$-functions as determinants or pfaffians associated with {\it outside decomposition} of shifted Young diagrams into {\it strips} (\cite{h96}).
In the study of multiple zeta function, Bachmann and Charlton proved general Jacobi-Trudi formulas for Schur multiple zeta functions for each {\it outside decomposition}.
In fact, they proved the Jacobi-Trudi formula for more general functions (\cite{bc}).

We first review the basic terminology of {\it outside decomposition} given by Hamel and Goulden (\cite{hg}).
For each box $\alpha$ of skew (shifted) diagram of $\lambda/\mu$, we define {\it content} of $\alpha$ as the quantity $j-i$ where $\alpha$ lies in row $i$ and in column $j$ of skew (shifted) diagram (referred to $(i,j)$ for convenient).
A strip in a skew shape diagram is a skew (shifted) diagram with an edgewise connected set of boxes that contains no $2\times2$ block of boxes. In other words, a strip has at most one box on each of its diagonals.
We say that the starting box of a strip is the box which is bottommost and leftmost in the strip and the ending box of a strip is the box which is topmost and rightmost in the strip.
\begin{defi}[Outside decomposition]
Suppose $(\theta_1,\ldots,\theta_r)$ are strips in a skew (shifted) diagram of $\lambda/\mu$ and each strip has a starting box on the left or bottom perimeter of the diagram and an ending box on the right or top perimeter of the diagram. Then if the disjoint union of these strips is the skew shape diagram of $\lambda/\mu$, we say the totally ordered set $(\theta_1,\ldots,\theta_r)$ is an outside decomposition of $\lambda/\mu$.
\end{defi}
\begin{exam}[$\lambda=(5,4,2,1)$]
\label{exout}
The following two are examples of outside decomposition $(\theta_1,\ldots,\theta_5)$ of $\lambda$.
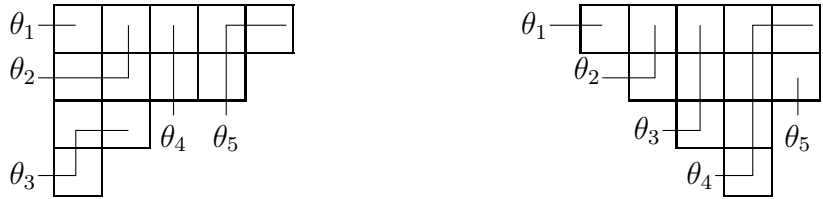
\begin{center}
\begin{figure}[h]
\begin{tikzpicture}
      \node at (3.5,1) {$\ydiagram{5,4,2,1}$};
    \node at (1.5,2) {$\theta_1$};
        \node at (1.5,1.4) {$\theta_2$};
        \node at (1.5,0) {$\theta_3$};
        \node at (3.5,0.5) {$\theta_4$};
        \node at (4.2,0.5) {$\theta_5$};
\draw (1.7,2) -- (2.2,2);
 \draw (1.7,1.3)--(2.9,1.3)--(2.9,2);
\draw (1.7,0)--(2.2,0)--(2.2,0.6) -- (2.9,0.6);
 \draw (3.5,0.8)--(3.5,2);
 \draw (4.2,0.8)--(4.2,2)--(5,2);
       \node at (10.5,1) {$\ydiagram{5,1+4,2+2,3+1}$};
    \node at (8.3,2) {$\theta_1$};
        \node at (9,1.4) {$\theta_2$};
        \node at (10.5,0) {$\theta_4$};
        \node at (9.8,0.6) {$\theta_3$};
        \node at (11.8,0.5) {$\theta_5$};
\draw (8.5,2) -- (9.2,2);
 \draw (9.2,1.3)--(9.9,1.3)--(9.9,2);
 \draw (10,0.6)--(10.5,0.6)--(10.5,2);
 \draw (10.7,0)--(11.2,0)--(11.2,2)--(12,2);
 \draw (11.8,0.8)--(11.8,1.3);
\end{tikzpicture}
\end{figure}
\end{center}
\end{exam}
We now define an additional operation $\theta_i\#\theta_j$ of strips $\theta_i$ and $\theta_j$. 
The following procedure is well-defined by \cite[Property 2.4]{ha}.
\begin{description}
    \item[Case.1] Suppose $\theta_i$ and $\theta_j$ have some boxes with the same content. Slide $\theta_i$ along
top-left-to-bottom-right diagonals so that the box of content $k$ in $\theta_i$ is superimposed on
the box of content $k$ in $\theta_j$ for all $k\in\mathbb Z$. 
We define $\theta_i\#\theta_j$ to be the diagram obtained from this superposition by taking all boxes
between the ending box of $\theta_i$ and the starting box of $\theta_j$ inclusive.
\item[Case.2] Suppose $\theta_i$ and $\theta_j$ are two disconnected pieces and thus do not have any
boxes of the same content. The starting box of one will be to the right and/or above
the ending box of the other. To bridge the gap between $\theta_i$ and $\theta_j$, insert boxes from the
ending box of one to the starting box of the other so that these inserted boxes follow
the approached-from-the-left or approached-from-below arrangement as do other boxes
of the same content in the outside decomposition. If there is a content such that there is no
box of that content in the diagram (and therefore no determination of the direction from
which the box is approached), then arbitrarily choose from which direction boxes of
this content should be approached, fix this choice for all boxes of the same content in
that particular diagram, and bridge the gap between $\theta_i$ and $\theta_j$ accordingly. Define $\theta_i\#\theta_j$
as in Case 1 with the following additional conventions: if the ending box of $\theta_i$ is edge
connected to the starting box of $\theta_j$, and occurs below or to the left of it, then $\theta_i\#\theta_j=\emptyset$; if
the ending box of $\theta_i$ is not edge connected but occurs below or to the left of the starting
box of $\theta_j$, $\theta_i\#\theta_j$ is undefined.
\end{description}
If $\pmb s=(s_{ij})$ satisfies $s_{ij}=s_{k\ell}$ with $i-j=k-\ell$, then we may define operation $\pmb s_{\lambda_i}\#\pmb s_{\lambda_j}$ of $\pmb s_{\lambda_i}$ and $\pmb s_{\lambda_j}$ in the same manner with the operation $\theta_i\#\theta_j$. We note that since $\pmb s=(s_{ij})$ have constant entries on the diagonals, this procedure is well-defined.
\begin{exam}
For the outside decomposition of Young diagram $\lambda$ in Example \ref{exout} (the left figure of the example),
\begin{align*}
    \theta_1\#\theta_2&=\ytableaushort{{-1}{0}},\ \theta_2\#\theta_1=\ytableaushort{{1},{0}},\ \theta_1\#\theta_4=\emptyset,\ \theta_4\#\theta_1=\ytableaushort{{2},{1},{0}},\\
    \theta_1\#\theta_5&\text{ is undefined, and } \theta_5\#\theta_1=\ytableaushort{{\none}{3}{4},{\none}{2},01},
\end{align*}
where the numbers indicate contents.
\end{exam}
\begin{exam}
Let \[\pmb s=\ytableaushort{{a_0}{a_1}{a_2}{a_3}{a_4},{\none}{a_0}{a_1}{a_2}{a_{3}},{\none}{\none}{a_0}{a_1},{\none}{\none}{\none}{a_0}}.\]
For the outside decomposition of shifted Young diagram $\lambda$ in Example \ref{exout} (the right figure of the example),
\begin{align*}
    \theta_1\#\theta_2&=\ytableaushort{{a_0}},\ \theta_2\#\theta_1=\ytableaushort{{a_1},{a_0}},\ \theta_1\#\theta_4=\ytableaushort{{a_0}},\ \theta_4\#\theta_1=\ytableaushort{{a_3}{a_4},{a_2},{a_1},{a_0}},\\
    \theta_1\#\theta_5&\text{ is undefined, and } \theta_5\#\theta_1=\ytableaushort{{a_3},{a_2},{a_1},{a_0}}.
\end{align*}
\end{exam}
Hamal (\cite{h96}) generalized classical pfaffian expression of Schur $Q$-function involving outside decompositions.
To explain the result, we extend the strips of our outside decomposition to the main diagonal, let $\rho$ be a strip consisting of a single box of content $0$ so that $\rho=\begin{ytableau}0\end{ytableau}$, where the number indicates the content. This allows us to define $\overline{\theta_i}=\theta_i\#\rho$. Let $(\overline\theta_p,\overline\theta_q)$ be formed by juxtaposing $\overline\theta_p$ and $\overline\theta_q$ with their boxes of content 0 lying on the
main diagonal with that of $\overline\theta_p$ immediately above and to the left of $\overline\theta_q$.
\begin{exam}
The $\overline{\theta}_p$ and $(\overline{\theta}_p,\overline{\theta}_q)$ of shifted Young diagram $\lambda$ in Exmaple \ref{exout} (the right figure of the example) are $\overline{\theta_p}=\theta_p$ for $1\le p\le 4$ and 
\begin{align*}
    (\overline\theta_1,\overline\theta_2)&=\ytableaushort{{0}{1},\none{0}},\ (\overline\theta_2,\overline\theta_1)=\ytableaushort{{1},{0},\none{0}},\ (\overline\theta_2,\overline\theta_4)=\ytableaushort{\none{3}{4},12,01,\none0},\ (\overline\theta_4,\overline\theta_2)=\ytableaushort{{3}{4},2,1,01,\none0},\\
    \overline\theta_5&=\ytableaushort{3,2,1,0},\ (\overline\theta_4,\overline\theta_5)=\ytableaushort{3,2{3}{4},12,01,\none0},\ (\overline\theta_5,\overline\theta_4)=\ytableaushort{{3}{4},23,12,01,\none0},
\end{align*}
where the numbers indicate contents.
\end{exam}
\begin{theorem}[cf. {\cite[Theorem 4.3]{fk},\cite[Theorem 1.4]{h96}}]
Let $\lambda$ and $\mu$ be strict partitions with $\mu\le\lambda$.
Let $\theta =(\theta_1, \theta_2,\ldots, \theta_k, \theta_{k+1},\ldots,\theta_r)$ be an outside decomposition of $SD(\lambda/\mu)$, where $\theta_p$ includes a box on the main
diagonal of $SD(\lambda/\mu)$ for $1\le p\le k$ and $\theta_p$ does not for $k+1\le p\le r$. If $k$ is odd, we replace $\theta$ by $(\emptyset,\theta_1,\ldots,\theta_r)$.  
Then the Schur $Q$-multiple zeta functions satisfy the identity
\[\zeta_{\lambda/\mu}^Q(\pmb s)={\rm pf}\begin{pmatrix}
\zeta_{(\overline{\theta}_p,\overline{\theta}_q)}^Q(\pmb s_{(\overline{\theta}_p,\overline{\theta}_q)})&\zeta_{\theta_i\#\theta_{r+k+1-j}}^Q(\pmb s_{\theta_i\#\theta_{r+k+1-j}})\\
-{}^t(\zeta_{\theta_i\#\theta_{r+k+1-j}}^Q(\pmb s_{\theta_i\#\theta_{r+k+1-j}}))&0
\end{pmatrix}\]
with $1 \le p, q \le k$ and $k + 1\le j \le r$. Here
\[\zeta_{(\overline{\theta}_p,\overline{\theta}_q)}^Q(\pmb s_{(\overline{\theta}_p,\overline{\theta}_q)})= -\zeta_{(\overline{\theta}_q,\overline{\theta}_p)}^Q(\pmb s_{(\overline{\theta}_q,\overline{\theta}_p)})\]
and $\zeta_{(\overline{\theta}_p,\overline{\theta}_p)}^Q(\pmb s_{(\overline{\theta}_p,\overline{\theta}_q)})=0$.
\end{theorem}
\section{Sum formula}
\label{sumformsec}
It is well known that the multiple zeta values of Euler-Zagier type satisfy the too many linear relations among multiple zeta values of Euler-Zagier type, such as the sum formula and duality formula. 
The following is the sum formula for multiple zeta values of Euler-Zagier type.
\begin{theorem}[Granville \cite{gr97}, Zagier] \label{originalsum}
 For positive integers $k$ and $r$ with $k>r$, we have
 \[
  \sum_{\substack{ k_{1}+\dots+k_{r}=k \\ k_1,\dots,k_{r-1}\ge1,k_{r}\ge2 }}
  \zeta(k_{1},\dots,k_{r})=\zeta(k).
 \]
\end{theorem}
As in the classical case, we prove the sum formula for Schur $P$- and $Q$-multiple zeta values.
\begin{theorem}
\label{sumpq}
 For positive integers $k$ and $r$ with $k>r$, we have
 \[
  \sum_{\substack{ k_{1}+\dots+k_{r}=k \\ k_1,\dots,k_{r-1}\ge1,k_{r}\ge2 }}
  \zeta_{(r)}^Q\left(\ytableaushort{{k_1}{\cdots}{k_r}}\right)=\sum_{i=1}^r2^i\binom{k-i-1}{r-i}\zeta(k)
 \]
 and
  \[
  \sum_{\substack{ k_{1}+\dots+k_{r}=k \\ k_1,\dots,k_{r-1}\ge1,k_{r}\ge2 }}
  \zeta_{(r)}^P\left(\ytableaushort{{k_1}{\cdots}{k_r}}\right)=\sum_{i=1}^r2^{i-1}\binom{k-i-1}{r-i}\zeta(k).
 \]
\end{theorem}
\begin{proof}
Let $\pmb k=\ytableaushort{{k_{1}}{\cdots}{k_{r}}}$.
By \eqref{pq}, it suffices to show the first identity. 
Example \ref{qtozeta} leads to
\begin{align*}
    \sum_{\substack{ k_{1}+\dots+k_{r}=k \\ k_1,\dots,k_{r-1}\ge1,k_{r}\ge2 }}
  \zeta_{(r)}^Q\left(\pmb k\right)&=\sum_{\substack{ k_{1}+\dots+k_{r}=k \\ k_1,\dots,k_{r-1}\ge1,k_{r}\ge2 }}\sum_{\pmb \ell\preceq_s\pmb k}2^{{\rm dep}({\pmb \ell})}\zeta(\pmb \ell)\\
  &=\sum_{\substack{ k_{1}+\dots+k_{r}=k \\ k_1,\dots,k_{r-1}\ge1,k_{r}\ge2 }}\sum_{\pmb \ell\preceq_s\pmb k}2^{{\rm dep}({\pmb \ell})}\zeta(\pmb \ell)\\
  &=\sum_{i=1}^r2^{i}\sum_{\substack{ k_{1}+\dots+k_{r}=k \\ k_1,\dots,k_{r-1}\ge1,k_{r}\ge2 }}\sum_{\substack{ \pmb \ell\preceq_s\pmb k\\ {\rm dep}(\pmb \ell)=i}}\zeta(\pmb \ell).
\end{align*}
For fixed $\pmb \ell$ with $|\pmb \ell|=k$ and ${\rm dep}(\pmb \ell)=i$, we count the number of $\pmb k$ with $ \pmb \ell\preceq_s\pmb k$ with $\pmb k\in ST((r),\mathbb Z)$. 
Since $\pmb k$ has to be admissible, it suffices to choose new $r-i$ division points of $\pmb \ell$ out of $(k-1)-(i-1)-1$ possibilities.
Therefore, 
\[\#\{\pmb k \in ST((r),\mathbb C)~|~ \pmb \ell\preceq_s\pmb k\}=\binom{k-i-1}{r-i}\]
and we have
\begin{align*}
    \sum_{\substack{ k_{1}+\dots+k_{r}=k \\ k_1,\dots,k_{r-1}\ge1,k_{r}\ge2 }}
  \zeta_{(r)}^Q\left(\pmb k\right)
  &=\sum_{i=1}^r2^{i}\binom{k-i-1}{r-i}\sum_{\substack{ |\pmb \ell|=k\\ {\rm dep}(\pmb \ell)=i}}\zeta(\pmb \ell).
  \intertext{The sum formula for multiple zeta values of Euler-Zagier type leads to}
  \sum_{\substack{ k_{1}+\dots+k_{r}=k \\ k_1,\dots,k_{r-1}\ge1,k_{r}\ge2 }}
  \zeta_{(r)}^Q\left(\pmb k\right)
  &=\sum_{i=1}^r2^i\binom{k-i-1}{r-i}\zeta(k).
\end{align*}
This proves the first identity.
Dividing both sides by $2$, we can confirm that the second identity holds. This completes the proof of the theorem.
\end{proof}
\begin{exam}For $(k,r)=(5,3)$, we have 
\[
  \sum_{\substack{ k_{1}+k_2+k_{3}=5 \\ k_1,k_2\ge1,k_{3}\ge2 }}
  \zeta_{(3)}^Q\left(\ytableaushort{{k_1}{k_2}{k_3}}\right)=11\zeta_{(1)}^Q\left(\ytableaushort{{5}}\right)=22\zeta(5).
 \]
\end{exam}
We have the following corollaries of Theorem \ref{sumpq}. The first is the sum formula in Schur $P$- or $Q$-multiple zeta values.
\begin{cor}
 For positive integers $k$ and $r$ with $k>r$, we have
 \[
  \sum_{\substack{ k_{1}+\dots+k_{r}=k \\ k_1,\dots,k_{r-1}\ge1,k_{r}\ge2 }}
  \zeta_{(r)}^Q\left(\ytableaushort{{k_1}{\cdots}{k_r}}\right)=\sum_{i=1}^r2^{i-1}\binom{k-i-1}{r-i}\zeta_{(1)}^Q\left(\ytableaushort{{k}}\right)
 \]
 and
  \[
  \sum_{\substack{ k_{1}+\dots+k_{r}=k \\ k_1,\dots,k_{r-1}\ge1,k_{r}\ge2 }}
  \zeta_{(r)}^P\left(\ytableaushort{{k_1}{\cdots}{k_r}}\right)=\sum_{i=1}^r2^{i-1}\binom{k-i-1}{r-i}\zeta_{(1)}^P\left(\ytableaushort{{k}}\right).
 \]
\end{cor}
The next corollary is the duality formula for a certain shape and weight. Before we explain a duality property of the Schur $Q$-multiple zeta function, we review the original duality formula for multiple zeta functions.
We denote a string $\underbrace{1,\ldots,1}_r$ of $1$'s by $\{1\}^r$.
Then for an admissible index
\[
{\boldsymbol k}=(\{1\}^{a_1-1}, b_1+1, \{1\}^{a_2-1}, b_2+1, \ldots, \{1\}^{a_m-1}, b_m+1)    
\]
with positive integers $a_1, b_1, a_2, b_2, \cdots, a_m, b_m\in {\mathbb Z}_{\geq 1}$, the following index is called {\it dual} index of ${\boldsymbol k}$:
\[{\boldsymbol k^{\dagger}}=(\{1\}^{b_m-1}, a_m+1, \{1\}^{b_{m-1}-1}, a_{m_1}+1, \ldots, \{1\}^{b_1-1}, a_1+1).
\]
The duality formula is the following.
\begin{theorem}[Duality formula \cite{z}]
\label{dualityformulaformzv}
 For any admissible index ${\boldsymbol k}=(k_1, \ldots, k_r)$ and its dual index ${\boldsymbol k^{\dagger}}=(k^{\dagger}_1., \ldots, k^{\dagger}_s)$, we have
\[
\zeta(k_1, \ldots, k_r)=\zeta(k^{\dagger}_1, \ldots, k^{\dagger}_s).
\]
\end{theorem}
As a special case of Theorem \ref{dualityformulaformzv}, it holds that
\[
\zeta(\{1\}^{k-2},2)=\zeta(k).
\]
Taking $\lambda=(k-1)$ and $\pmb k=\ytableaushort{{1}{\cdots}{1}{2}}\in ST(\lambda,\mathbb C)$, we have the following formula similar to the above identity.
\begin{cor}
\label{dual}
 For positive integers $k$, we have
 \[
  \zeta_{(k-1)}^Q\left(\ytableaushort{{1}{\cdots}{1}{2}}\right)=(2^{k-1}-1)\zeta_{(1)}^Q\left(\ytableaushort{{k}}\right)=(2^{k}-2)\zeta(k)
 \]
 and
  \[
  \zeta_{(k-1)}^P\left(\ytableaushort{{1}{\cdots}{1}{2}}\right)=(2^{k-1}-1)\zeta_{(1)}^P\left(\ytableaushort{{k}}\right)=(2^{k-1}-1)\zeta(k).
 \]
\end{cor}
\section{Symplectic Schur Multiple zeta functions}
\label{symplecticsection}
First of all, we review the basic terminology to define symplectic or orthogonal Schur multiple zeta functions.
We associate partition $\lambda$ with {\it the Young diagram}
\[D(\lambda)=\{(i, j)\in {\mathbb Z}^2 ~|~ 1\leq i\leq r, i\leq j\leq \lambda_i\}\] depicted as a collection of square boxes with the $i$-th row has $\lambda_i$ boxes.
For a partition $\lambda$, a Young tableau $(t_{ij})$ of shape $\lambda$ over a set $X$ is a filling of $D(\lambda)$ with $t_{ij}\in X$ into $(i,j)$ box of $D(\lambda)$.
We denote by $T(\lambda,X)$ the set of all Young tableaux of shape $\lambda$ over $X$.

Let $[\overline{N}]$ be the set $\{1,\overline{1},2,\overline{2},\ldots,N,\overline{N}\}$ with the total ordering $1<\overline{1}<2<\overline{2}<\cdots<N<\overline{N}$. Then, a {\it symplectic tableau} $\pmb t=(t_{ij})\in T(\lambda,[\overline{N}])$ is obtained by numbering all the boxes of $D(\lambda)$ with letters from $[\overline{N}]$ such that
\begin{description}
    \item[SP1] the entries of $\pmb t$ are weakly increasing along each row of $\pmb t$,
    \item[SP2] the entries of $\pmb t$ are strictly increasing down each column of $\pmb t$,
    \item[SP3] the boxes of content $-i$ contain entries which are greater than or equal to $i+1$.
\end{description}
We refer to the third condition SP3 as the symplectic condition.
We denote by $SP_N(\lambda)$ the set of symplectic tableaux of shape $\lambda$.
Then for a given set ${ \pmb s}=(s_{ij})\in T(\lambda,\mathbb{C})$ of variables, 
{\it symplectic Schur multiple zeta functions} of {shape} $\lambda$ are defined as
\begin{equation}
\label{symplecticdef}
\zeta_\lambda^{{\rm sp},N}({ \pmb s})=\sum_{M\in SP_N(\lambda)}\frac{1}{M^{ \pmb s}},
\end{equation}
where $M^{ \pmb s}=\displaystyle{\prod_{(i, j)\in D(\lambda)}|m_{ij}|^{s_{ij}}}$ for $M=(m_{ij})\in SP_N(\lambda)$ and $|i|=i,|\overline{i}|=i^{-1}$.

Hamel composed a directed graph $D$ corresponding to the symplectic Schur functions \cite{ha} and applied the Stembridge Theorem (\cite{st}) to obtain the following determinant expression of the symplectic Schur functions.
\begin{theorem}[{\cite[Theorem 3.1]{ha}}]
Let $\lambda/\mu$ be a partition of skew type. Then, for any outside decomposition $(\theta_1,\ldots,\theta_r)$ of $\lambda/\mu$,
\[{\rm sp}_{\lambda/\mu}={\rm det}({\rm sp}_{\theta_i\#\theta_j})_{1\le i,j\le r}.\]
\end{theorem}
Following the Hamel way, we compose a directed graph $D$ corresponding to the symplectic Schur multiple zeta functions.
For a fixed positive integer $N$, we begin with the $y$-axis labeled by $1,\overline{1},2,\overline{2},\ldots, N, \overline{N}$ and direct an edge $u\rightarrow v$ whenever $u-v = (0, 1), (0,-1),(1, 0)$, or $(1, -1)$. We add four restrictions: 
 a down-vertical step must not precede an up-vertical step, an up-vertical step must not precede a down-vertical step, a down-vertical step must not precede a horizontal step, and an up-vertical step must not precede a diagonal step. Because of the symplectic condition, we add a left boundary in the form
of a ``backwards lattice path'' from $(0,1)$ to $(0,\overline{1})$ to $(0,2)$ to $(-1,2)$ to $(-1,\overline{2})$ to $(-1,3)$ to $(-2,3)$ to $(-2,\overline{3})$ to $(-2,4)$ to $(-3,4)\ldots$. We indicate this left boundary by the dotted line in Figure \ref{symplecticleft}.

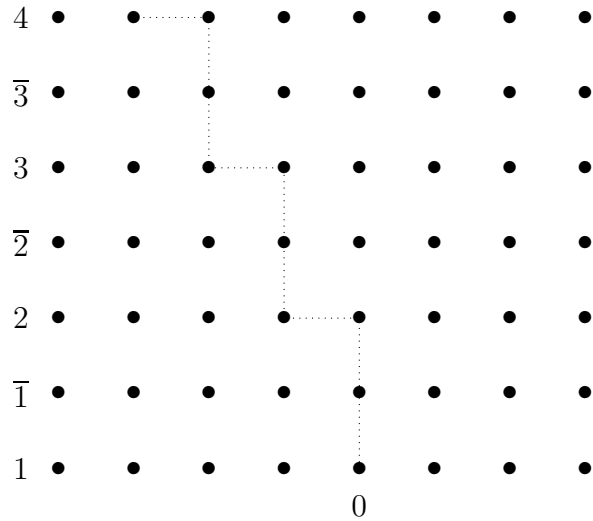
\begin{figure}[ht]
\begin{center}
 \begin{tikzpicture} 
    \node at (0,1) {$\bullet$};
   \node at (0,2) {$\bullet$};
   \node at (0,3) {$\bullet$};
   \node at (0,4) {$\bullet$}; 
   \node at (0,5) {$\bullet$};
   \node at (0,6) {$\bullet$}; 
   \node at (1,1) {$\bullet$};
   \node at (1,2) {$\bullet$};
   \node at (1,3) {$\bullet$};
   \node at (1,4) {$\bullet$};    
   \node at (2,1) {$\bullet$};
   \node at (2,2) {$\bullet$};
   \node at (2,3) {$\bullet$};  
   \node at (2,4) {$\bullet$};  
   \node at (3,1) {$\bullet$};
   \node at (3,2) {$\bullet$};
   \node at (3,3) {$\bullet$};  
  \node at (3,4) {$\bullet$};  
   \node at (4,1) {$\bullet$};
   \node at (4,2) {$\bullet$};
   \node at (4,3) {$\bullet$};  
   \node at (4,4) {$\bullet$};  
   \node at (5,1) {$\bullet$};
   \node at (5,2) {$\bullet$};
   \node at (5,3) {$\bullet$};  
   \node at (5,4) {$\bullet$};  
   \node at (6,1) {$\bullet$};
   \node at (6,2) {$\bullet$};
   \node at (6,3) {$\bullet$};  
   \node at (6,4) {$\bullet$};  
    \node at (7,1) {$\bullet$};
   \node at (7,2) {$\bullet$};
   \node at (7,3) {$\bullet$};  
   \node at (7,4) {$\bullet$}; 
   \node at (1,5) {$\bullet$};
   \node at (2,5) {$\bullet$};  
   \node at (3,5) {$\bullet$}; 
   \node at (4,5) {$\bullet$};
   \node at (5,5) {$\bullet$};
   \node at (6,5) {$\bullet$};  
   \node at (7,5) {$\bullet$}; 
   \node at (1,6) {$\bullet$};
   \node at (2,6) {$\bullet$};  
   \node at (3,6) {$\bullet$}; 
   \node at (4,6) {$\bullet$};
   \node at (5,6) {$\bullet$};
   \node at (6,6) {$\bullet$};  
   \node at (7,6) {$\bullet$};
    \node at (0,7) {$\bullet$};
   \node at (1,7) {$\bullet$};
   \node at (2,7) {$\bullet$};  
   \node at (3,7) {$\bullet$}; 
   \node at (4,7) {$\bullet$};
   \node at (5,7) {$\bullet$};
   \node at (6,7) {$\bullet$};  
   \node at (7,7) {$\bullet$};
    \node at (-0.5,7) {$4$};
    \node at (-0.5,6) {$\overline{3}$};   
    \node at (-0.5,5) {$3$};
    \node at (-0.5,4) {$\overline{2}$};
    \node at (-0.5,3) {$2$};
    \node at (-0.5,2) {$\overline{1}$};
     \node at (-0.5,1) {$1$};
     \node at (4,0.5) {$0$};
  \draw[dotted] (4,1) -- (4,3) -- (3,3) --(3,5)-- (2,5)--(2,7)--(1,7);
 \end{tikzpicture}
\end{center}
\caption{Left boundary given by the symplectic condition}
 \label{symplecticleft}
\end{figure}
In the following, we may omit this left boundary for simplicity. 

For a fixed outside decomposition $(\theta_1,\ldots,\theta_r)$ of $\lambda/\mu$, we will construct a non-intersecting $r$-tuple of lattice paths that corresponds to a symplectic tableau of shape $\lambda/\mu$ with the outside decomposition $(\theta_1,\ldots,\theta_r)$, such that the $i$-th path corresponds to the $i$-th strip and begins at $B_i$ and ends at $E_i$ as described now. Fix points $B_i = (t-s,-(t-s)+1)$ if the $i$-th strip has the starting box on left perimeter in a box $(s,t)$ of the diagram and if $t-s\le0$ (i.e. $B_i$ is on the left boundary), or $B_i = (t-s,1)$ if the $i$-th strip has the starting box on the left perimeter in a box $(s,t)$ of the diagram and if $t-s>0$, or $B_i = (t-s,\infty)$ if the $i$-th strip has the starting box on the bottom perimeter in a box $(s,t)$ of the diagram ($B_i = (t-s,\infty)$ if both). Fix points $E_i = (v-u+1,1)$ if the $i$-th strip has the ending box on the top perimeter in box $(u,v)$ of the diagram, or $E_i = (v-u+1,\infty)$ if the $i$-th strip has the ending box on the right perimeter in a box $(u,v)$ of the diagram ($E_i = (v-u+1,\infty)$ if both).

For the $j$-th strip construct a path starting at $B_j$ (called the starting point) and ending at $E_j$ (called the ending point) as follows: if a box containing $i$ (resp. $\overline{i}$) and at coordinates $(a,b)$ in the diagram is approached from the left in the strip, put a horizontal step from
$(b-a, i)$ to $(b-a+1, i)$ (resp. $(b-a, \overline{i})$ to $(b-a+1, \overline{i})$); if a box containing $i$ (resp. $\overline{i}$) and at coordinates $(a,b)$ in the diagram is approached from below in the strip, put a diagonal step from $(b-a, \overline{i})$ to $(b-a+1, i)$ (resp. $(b-a, i+1)$ to $(b-a+1, \overline{i})$). 
We note that the physical locations of the termination points of the steps are independent of the outside decomposition and depend only on the contents of the boxes. See Figure \ref{symplecticfigure} in which first the ending points of steps are shown alone and then complete paths for two different outside decompositions are shown. We note that no two paths can have the same starting and/or ending points, since that would imply two boxes of the same content on the same section of the perimeter. Connect these non-vertical steps with vertical steps. It is a routine to verify that there is a unique.
Under the above settings, Hamel showed that the symplectic tableaux of shape $\lambda/\mu$ can be identified with the non-intersecting paths in $\mathscr P_{0}((B_i);(E_i))$, and $(B_i)$ is $D$-compatible with $(E_i)$. 

We next define the weight of each step.
Let $v_i(P)=(v_{ij}(P))_{j=0}$ be the sequence of vertices representing the path $P\in \mathscr P_{0}(B_i;E_i)$. Successively, let $v_i^w(P)=(v_{ij}^w(P))_{j=1}$ be the sub-sequence of $v_i(P)$ with $v_{ij}(P)-v_{i(j-1)}(P)= (1,0)$ or $(1,-1)$, let $v_i^1(P)=(v_{ij}^1(P))_{j=1}$ be the sub-sequence of $v_i(P)$ with $v_{ij}(P)-v_{i(j-1)}(P)= (0,1)$ or $(0,-1)$.
For $\pmb s\in T(\lambda,\mathbb C)$, we assign the weight $w(v_{ij}^w(P))=|v_2|^{-s_{j}(\theta_i)}$ to $v_{ij}^w(P)=(v_1,v_2)$ with the $j$-th component $s_j(\theta_i)$ of $\theta_i$ and we assign the weight $w(v_{ij}^1(P))=1$ to $v_{ij}^1(P)$. 
Then, we define \[w(P)=\prod_{v_{ij}(P)}w(v_{ij}(P)),\]
and for $r$-tuple of non-intersecting paths of $(P_1,\ldots,P_r)$ with $P_i\in \mathscr P(B_i;E_i)$,
\[w(P_1,\ldots,P_r)=\prod_{i=1}^rw(P_i).\]
Then, due to the Hamel composition, we find that
\[\zeta_\lambda^{{\rm sp},N}(\pmb s)=\sum_{P_i\in \mathscr P(B_i;E_i)}w(P_1,\ldots,P_r).\]
\begin{exam}
\label{72}
For $\lambda=(5,3,3,1)$ let $(s_{ij})\in T(\lambda,\mathbb C)$. The weight $w(P_i)$ are
\begin{align*}
    w(P_1)&=1^{s_{11}},&&w(P_2)=\frac{2^{s_{21}}}{3^{s_{22}}2^{s_{12}}},\\
    w(P_3)&=\frac{3^{s_{32}}3^{s_{33}}}{4^{s_{41}}3^{s_{31}}3^{s_{23}}2^{s_{13}}},&&w(P_4)=\frac{3^{s_{24}}3^{s_{14}}}{4^{s_{15}}},
\end{align*} 
and the corresponding symplectic tableau is
\[\ytableaushort{{\overline{1}}22{\overline{2}}{4},{\overline{2}}33,3{\overline{3}}{\overline{3}},{4}}.\]
\begin{figure}[h]
\begin{center}
 \begin{tikzpicture} 
    \node at (8,1) {$\bullet$};
   \node at (8,2) {$\bullet$};
   \node at (8,3) {$\bullet$};
   \node at (8,4) {$\bullet$}; 
   \node at (8,5) {$\bullet$};
   \node at (8,6) {$\bullet$}; 
   \node at (1,1) {$\bullet$};
   \node at (1,2) {$\bullet$};
   \node at (1,3) {$\bullet$};
   \node at (1,4) {$\bullet$};    
   \node at (2,1) {$\bullet$};
   \node at (2,2) {$\bullet$};
   \node at (2,3) {$\bullet$};  
   \node at (2,4) {$\bullet$};  
   \node at (3,1) {$\bullet$};
   \node at (3,2) {$\bullet$};
   \node at (3,3) {$\bullet$};  
  \node at (3,4) {$\bullet$};  
   \node at (4,1) {$\bullet$};
   \node at (4,2) {$\bullet$};
   \node at (4,3) {$\bullet$};  
   \node at (4,4) {$\bullet$};  
   \node at (5,1) {$\bullet$};
   \node at (5,2) {$\bullet$};
   \node at (5,3) {$\bullet$};  
   \node at (5,4) {$\bullet$};  
   \node at (6,1) {$\bullet$};
   \node at (6,2) {$\bullet$};
   \node at (6,3) {$\bullet$};  
   \node at (6,4) {$\bullet$}; 
   \node at (8,7) {$\bullet$};
    \node at (7,1) {$\bullet$};
   \node at (7,2) {$\bullet$};
   \node at (7,3) {$\bullet$};  
   \node at (7,4) {$\bullet$}; 
   \node at (1,5) {$\bullet$};
   \node at (2,5) {$\bullet$};  
   \node at (3,5) {$\bullet$}; 
   \node at (4,5) {$\bullet$};
   \node at (5,5) {$\bullet$};
   \node at (6,5) {$\bullet$};  
   \node at (7,5) {$\bullet$}; 
   \node at (1,6) {$\bullet$};
   \node at (2,6) {$\bullet$};  
   \node at (3,6) {$\bullet$}; 
   \node at (4,6) {$\bullet$};
   \node at (5,6) {$\bullet$};
   \node at (6,6) {$\bullet$};  
   \node at (7,6) {$\bullet$};
   \node at (1,7) {$\bullet$};
   \node at (2,7) {$\bullet$};  
   \node at (3,7) {$\bullet$}; 
   \node at (4,7) {$\bullet$};
   \node at (5,7) {$\bullet$};
   \node at (6,7) {$\bullet$};  
   \node at (7,7) {$\bullet$};
    \node at (0.5,7) {$4$};
    \node at (0.5,6) {$\overline{3}$};   
    \node at (0.5,5) {$3$};
    \node at (0.5,4) {$\overline{2}$};
    \node at (0.5,3) {$2$};
    \node at (0.5,2) {$\overline{1}$};
     \node at (0.5,1) {$1$};
     \node at (8,0.5) {$4$};
     \node at (4,0.5) {$0$};
     \node at (3,0.5) {$-1$};
     \node at (2,0.5) {$-2$};
     \node at (1,0.5) {$-3$};
         \node at (3.5,1) {$B_1$};
    \node at (2.5,3) {$B_2$};   
    \node at (2,7.5) {$B_3$};
    \node at (6,7.5) {$B_4$};
    \node at (5,0.5) {$E_1$};
    \node at (6,0.5) {$E_2$};
     \node at (7,0.5) {$E_3$};
     \node at (8.5,7.5) {$E_4$};
 \draw[dotted] (4,1) -- (4,3) -- (3,3) --(3,5)-- (2,5)--(2,7)--(1,7);
 \draw (4,1)--(4,2) -- (5,2) -- (5,1);
\draw[loosely dashdotted] (3,3) -- (3,4) -- (4,4) -- (4,5) -- (5,5)--(5,4)--(6,3)--(6,1); 
\draw[dashdotted] (1.5,7.5) -- (2,7) -- (2,6) -- (3,5)--(3,6)--(5,6)--(6,5)--(6,4)--(7,3)--(7,1);
  \draw[dashed] (5.5,7.5) --(6,7)-- (6,6) -- (7,5)--(7,7)--(8,7) --(8,7.5);
 \end{tikzpicture}
\end{center}
\caption{$(P_1,P_2,P_3,P_4)$ in Example \ref{72}}
\label{symplecticfigure}
\end{figure}
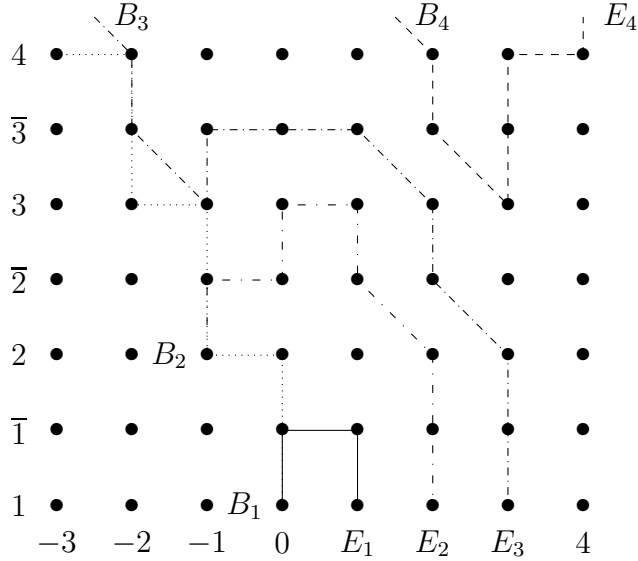
\end{exam}

\begin{theorem}
Let $\lambda=(\lambda_1,\ldots,\lambda_{r})$, $\mu=(\mu_1,\ldots,\mu_{s})$ be partitions. 
Then for $\pmb s\in T^{\rm diag}(\lambda/\mu,\mathbb C)$ and any outside decomposition $(\theta_1,\ldots,\theta_r)$ of $\lambda/\mu$,
\[\zeta_{\lambda/\mu}^{{\rm sp},N}(\pmb s)={\rm det}(\zeta_{\theta_i\#\theta_j}^{{\rm sp},N}(\pmb s_{(\lambda_i,\lambda_j)}))_{1\le i,j\le r}
,\]
where $\pmb s_{(\lambda_i,\lambda_j)}=\pmb s_{\lambda_i}\#\pmb s_{\lambda_j}$.
\end{theorem}
\begin{exam}
\label{74}
Let $\lambda=(3,2)$ and its outside decomposition $(\theta_1,\theta_2)$ be depicted as 
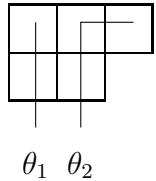
\begin{center}
\begin{figure}[h]
\begin{tikzpicture}
      \node at (3.5,1) {$\ydiagram{3,2}$};
    \node at (2.9,-0.5) {$\theta_1$};
        \node at (3.5,-0.5) {$\theta_2$};
 \draw (2.9,0)--(2.9,1.4);
 \draw (3.5,0)--(3.5,1.4)--(4.2,1.4);
\end{tikzpicture}.
\end{figure}
\end{center}
Then, if $(a_{j-i})=(s_{ij})\in T^{\rm diag}(\lambda,\mathbb C)$,
\begin{align*}
    \zeta_\lambda^{{\rm sp},N}(\pmb s)=&{\rm det}\begin{pmatrix}
\zeta_{\theta_1}^{{\rm sp},N}\left(\ytableaushort{{a_{0}},{a_{-1}}}\right)&\zeta_{\theta_1\#\theta_2}^{{\rm sp},N}\left(\ytableaushort{{a_{0}}}\right)\\
\zeta_{\theta_2\#\theta_1}^{{\rm sp},N}\left(\ytableaushort{{a_1}{a_2},{a_{0}},{a_{-1}}}\right)&\zeta_{\theta_2}^{{\rm sp},N}\left(\ytableaushort{{a_1}{a_2},{a_0}}\right)
\end{pmatrix}\\
=&\zeta_{\theta_1}^{{\rm sp},N}\left(\ytableaushort{{a_{0}},{a_{-1}}}\right)\zeta_{\theta_2}^{{\rm sp},N}\left(\ytableaushort{{a_1}{a_2},{a_0}}\right)-\zeta_{\theta_1\#\theta_2}^{{\rm sp},N}\left(\ytableaushort{{a_{0}}}\right)\zeta_{\theta_2\#\theta_1}^{{\rm sp},N}\left(\ytableaushort{{a_1}{a_2},{a_{0}},{a_{-1}}}\right).
\end{align*}
\end{exam}
\begin{remark}
  \label{contentrem}
  The function in Example \ref{74} satisfies \[\zeta_{\theta_2\#\theta_1}^{{\rm sp},N}\left(\ytableaushort{{a_1}{a_2},{a_{0}},{a_{-1}}}\right)\neq\zeta_{(2,1,1)}^{{\rm sp},N}\left(\ytableaushort{{a_1}{a_2},{a_{0}},{a_{-1}}}\right)\]
  in general. We note that for $i=-1,0,1,2$ the contents of each $\ytableaushort{{a_i}}$ are not the same.
\end{remark}

\section{Orthogonal Schur Multiple zeta functions}
\label{orthogonalsec}
Hamel also composed a directed graph $D$ corresponding to the orthogonal Schur functions \cite{ha} and obtain the determinant expression of the orthogonal Schur functions. As in Section \ref{symplecticsection}, we compose a directed graph $D$ corresponding to the orthogonal Schur multiple zeta functions. As in Section \ref{symplecticsection}, we prove results corresponding to the following Hamel result.

We define orthogonal Schur multiple zeta functions.
Let $[\overline{N}]^{\infty}$ be the set $\{1,\overline{1},2,\overline{2},\ldots,$ $N,\overline{N},\infty\}$ with the total ordering $1<\overline{1}<2<\overline{2}<\cdots<N<\overline{N}<\infty$. 
For a fixed partition $\lambda$, a {\it so-tableau} $\pmb t=(t_{ij})\in T(\lambda,[\overline{N}]^{\infty})$ is obtained by numbering all the boxes of $D(\lambda)$ with letters from $[\overline{N}]^{\infty}$ such that
\begin{description}
    \item[SO1] the entries of $\pmb t$ are weakly increasing along each row of $\pmb t$,
    \item[SO2] the entries of $\pmb t$ are strictly increasing down each column of $\pmb t$,
    \item[SO3] the boxes of content $-i$ contain entries which are greater than or equal to $i+1$,
    \item[SO4] the entries equal to $\infty$ form a shape which is such that no two symbols $\infty$ appear in the same row.
\end{description}
One may find that the conditions SO1-SO3 are the same as SP1-SP3.
We denote by $SO_N(\lambda)$ the set of so-tableaux of shape $\lambda$.
Then for a given set ${ \pmb s}=(s_{ij})\in T(\lambda,\mathbb{C})$ of variables, 
{\it orthogonal Schur multiple zeta functions} of {shape} $\lambda$ are defined as
\begin{equation}
\label{orthogonaldef}
\zeta_\lambda^{{\rm so},N}({ \pmb s})=\sum_{M\in SP_N(\lambda)}\frac{1}{M^{ \pmb s}},
\end{equation}
where $M^{ \pmb s}=\displaystyle{\prod_{(i, j)\in D(\lambda)}|m_{ij}|^{s_{ij}}}$ for $M=(m_{ij})\in SO_N(\lambda)$ and $|i|=i,|\overline{i}|=i^{-1}$ for any integer $i$ and $|\infty|=1$.
We note that the $\infty$ contributes $1$ to the weight of the tableau. Therefore, they are ``dummy elements'' in a sense.
\begin{theorem}[{\cite[Theorem 3.2]{ha}}]
Let $\lambda/\mu$ be a partition of skew type. Then, for any outside decomposition $(\theta_1,\ldots,\theta_r)$ of $\lambda/\mu$,
\[{\rm so}_{\lambda/\mu}={\rm det}({\rm so}_{\theta_i\#\theta_j})_{1\le i,j\le r}.\]
\end{theorem}
As in the symplectic Schur multiple zeta functions, we consider the $y$-axis with $1,\overline{1},2,\overline{2},\ldots,N,\overline{N},\infty$. We define lattice paths with five types of permissible steps. These steps are the four steps in Section \ref{symplecticsection}, and up-diagonal steps from height $\overline{N}$ to height $\infty$ that increase the $x$-coordinate by $1$ and increase the $y$-coordinate by $1$. We distinguish between horizontal
steps at integer levels and horizontal steps at $\infty$. The steps are subject to the same restrictions
as in Section \ref{symplecticsection} plus the following additional restrictions: an up-vertical step must not
precede a horizontal step at $\infty$, and a down-vertical step must not precede an
up-diagonal step. We also require that all steps between lines $x=c$ and $x=c+1$ for
all $c$ are either 
\begin{enumerate}
    \item horizontal at $\infty$ or down-diagonal, or
    \item horizontal at integer levels or up-diagonal.
\end{enumerate}
The determination of whether the steps are of type (1) or (2) depends on the outside decomposition: if boxes of content $c$ are approached from the left, then steps between $x=c$ and $x=c+1$ must be of type (2); if the boxes of content $c$ are approached from below, then steps between $x=c$ and $x=c+1$ must be of type (1). We fix beginning points $B_i$ and ending points $E_i$ as in Section \ref{symplecticsection} with the adjustment that the $y$-coordinate of the highest points is $\infty+1$ instead of $\infty$. Given $\pmb s\in SO(\lambda/\mu,\mathbb C)$ with an outside decomposition, we can construct an $r$-tuple of non-intersecting lattice paths. For each strip construct a path as follows: if a box contains $i$ or $\overline{i}$, place a step as in the proof of Section \ref{symplecticsection}. If a box contains $\infty$, is at coordinates $(a,b)$ in the diagram, and is approached from the left in the
strip, put an up-diagonal step from $(a-b, \overline{N})$ to $(a-b+1,\infty)$; if it is approached from below, put a horizontal step from $(a-b, \infty)$ to $(a-b+1,\infty)$. We connect these non-vertical paths with vertical paths.

We can define the weight of each path.
Let $v_i(P)=(v_{ij}(P))_{j=0}$ be the sequence of vertices representing the path $P\in \mathscr P_{0}(B_i;E_i)$. Successively, let $v_i^w(P)=(v_{ij}^w(P))_{j=1}$ be the sub-sequence of $v_i(P)$ with $v_{ij}(P)-v_{i(j-1)}(P)= (1,0)$ or $(1,-1)$, let $v_i^1(P)=(v_{ij}^1(P))_{j=1}$ be the sub-sequence of $v_i(P)$ with $v_{ij}(P)-v_{i(j-1)}(P)= (0,1)$ or $(0,-1)$.
For $\pmb s\in T(\lambda,\mathbb C)$, we assign the weight \[w(v_{ij}^w(P))=\left\{\begin{array}{cc}|v_2|^{-s_{j}(\theta_i)}&\text{ if }v_2\neq\infty,\\
1&\text{ if }v_2=\infty
\end{array}\right.\] to $v_{ij}^w(P)=(v_1,v_2)$ with the $j$-th component $s_j(\theta_i)$ of $\theta_i$ and we assign the weight $w(v_{ij}^1(P))=1$ to $v_{ij}^1(P)$. 
Then, we define \[w(P)=\prod_{v_{ij}(P)}w(v_{ij}(P))\]
and for an $r$-tuple of non-intersecting paths of $(P_1,\ldots,P_r)$ with $P_i\in \mathscr P(B_i;E_i)$,
\[w(P_1,\ldots,P_r)=\prod_{i=1}^rw(P_i).\]
Then, due to the Hamel composition, we find that
\[\zeta_\lambda^{{\rm so},N}(\pmb s)=\sum_{P_i\in \mathscr P(B_i;E_i)}w(P_1,\ldots,P_r).\]
\begin{exam}
\label{82}For $\lambda=(5,3,3,1)$, let $(s_{ij})\in T(\lambda,\mathbb C)$. The weight $w(P_i)$ are
\begin{align*}
    w(P_1)&=1^{s_{11}},&&w(P_2)=\frac{2^{s_{21}}}{3^{s_{22}}2^{s_{12}}},\\
    w(P_3)&=\frac{3^{s_{32}}3^{s_{33}}}{3^{s_{31}}3^{s_{23}}2^{s_{13}}},&&w(P_4)=3^{s_{24}}3^{s_{14}},
\end{align*} 
and the corresponding symplectic tableau is
\[\ytableaushort{{\overline{1}}22{\overline{2}}{\infty},{\overline{2}}33,3{\overline{3}}{\overline{3}},{\infty}}.\]
\begin{figure}[h]
\begin{center}
 \begin{tikzpicture} 
    \node at (8,1) {$\bullet$};
   \node at (8,2) {$\bullet$};
   \node at (8,3) {$\bullet$};
   \node at (8,4) {$\bullet$}; 
   \node at (8,5) {$\bullet$};
   \node at (8,6) {$\bullet$}; 
   \node at (1,1) {$\bullet$};
   \node at (1,2) {$\bullet$};
   \node at (1,3) {$\bullet$};
   \node at (1,4) {$\bullet$};    
   \node at (2,1) {$\bullet$};
   \node at (2,2) {$\bullet$};
   \node at (2,3) {$\bullet$};  
   \node at (2,4) {$\bullet$};  
   \node at (3,1) {$\bullet$};
   \node at (3,2) {$\bullet$};
   \node at (3,3) {$\bullet$};  
  \node at (3,4) {$\bullet$};  
   \node at (4,1) {$\bullet$};
   \node at (4,2) {$\bullet$};
   \node at (4,3) {$\bullet$};  
   \node at (4,4) {$\bullet$};  
   \node at (5,1) {$\bullet$};
   \node at (5,2) {$\bullet$};
   \node at (5,3) {$\bullet$};  
   \node at (5,4) {$\bullet$};  
   \node at (6,1) {$\bullet$};
   \node at (6,2) {$\bullet$};
   \node at (6,3) {$\bullet$};  
   \node at (6,4) {$\bullet$}; 
   \node at (8,7) {$\bullet$};
    \node at (7,1) {$\bullet$};
   \node at (7,2) {$\bullet$};
   \node at (7,3) {$\bullet$};  
   \node at (7,4) {$\bullet$}; 
   \node at (1,5) {$\bullet$};
   \node at (2,5) {$\bullet$};  
   \node at (3,5) {$\bullet$}; 
   \node at (4,5) {$\bullet$};
   \node at (5,5) {$\bullet$};
   \node at (6,5) {$\bullet$};  
   \node at (7,5) {$\bullet$}; 
   \node at (1,6) {$\bullet$};
   \node at (2,6) {$\bullet$};  
   \node at (3,6) {$\bullet$}; 
   \node at (4,6) {$\bullet$};
   \node at (5,6) {$\bullet$};
   \node at (6,6) {$\bullet$};  
   \node at (7,6) {$\bullet$};
   \node at (1,7) {$\bullet$};
   \node at (2,7) {$\bullet$};  
   \node at (3,7) {$\bullet$}; 
   \node at (4,7) {$\bullet$};
   \node at (5,7) {$\bullet$};
   \node at (6,7) {$\bullet$};  
   \node at (7,7) {$\bullet$};
    \node at (0.5,7) {$\infty$};
    \node at (0.5,6) {$\overline{3}$};   
    \node at (0.5,5) {$3$};
    \node at (0.5,4) {$\overline{2}$};
    \node at (0.5,3) {$2$};
    \node at (0.5,2) {$\overline{1}$};
     \node at (0.5,1) {$1$};
     \node at (4,0.5) {$0$};
          \node at (8,0.5) {$4$};
     \node at (4,0.5) {$0$};
     \node at (3,0.5) {$-1$};
     \node at (2,0.5) {$-2$};
     \node at (1,0.5) {$-3$};
         \node at (3.5,1) {$B_1$};
    \node at (2.5,3) {$B_2$};   
    \node at (2,7.5) {$B_3$};
    \node at (6,7.5) {$B_4$};
    \node at (5,0.5) {$E_1$};
    \node at (6,0.5) {$E_2$};
     \node at (7,0.5) {$E_3$};
     \node at (8.5,7.5) {$E_4$};
 \draw[dotted] (4,1) -- (4,3) -- (3,3) --(3,5)-- (2,5)--(2,7);
 \draw (4,1)--(4,2) -- (5,2) -- (5,1);
\draw[loosely dashdotted] (3,3) -- (3,4) -- (4,4) -- (4,5) -- (5,5)--(5,4)--(6,3)--(6,1); 
\draw[dashdotted] (1.5,7.5) -- (2,7) -- (2,6) -- (3,5)--(3,6)--(5,6)--(6,5)--(6,4)--(7,3)--(7,1);
  \draw[dashed] (5.5,7.5) --(6,7)-- (6,6) -- (7,5)--(7,7)--(8,7) --(8,7.5);
 \end{tikzpicture}
\end{center}
\caption{$(P_1,P_2,P_3,P_4)$ in Example \ref{82}}
\label{orthogonalfigure}
\end{figure}
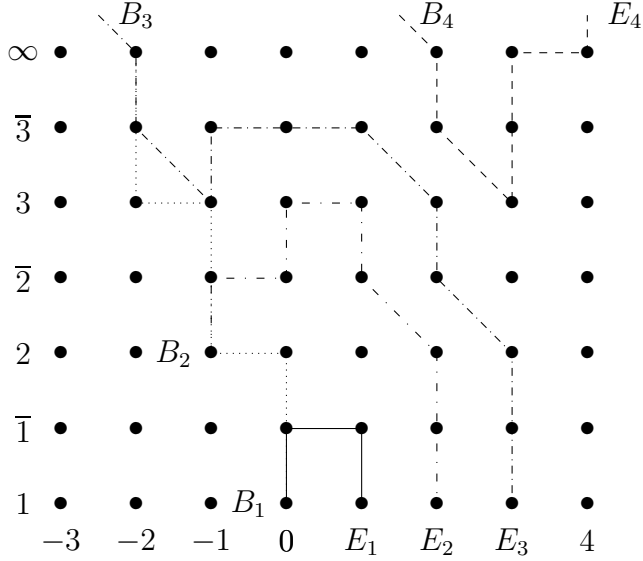
\end{exam}
\begin{theorem}
Let $\lambda=(\lambda_1,\ldots,\lambda_{r})$, $\mu=(\mu_1,\ldots,\mu_{s})$ be partitions. 
Then for $\pmb s\in T^{\rm diag}(\lambda/\mu,\mathbb C)$ and any outside decomposition $(\theta_1,\ldots,\theta_r)$ of $\lambda/\mu$,
\[\zeta_{\lambda/\mu}^{{\rm so},N}(\pmb s)={\rm det}(\zeta_{\theta_i\#\theta_j}^{{\rm so},N}(\pmb s_{(\lambda_i,\lambda_j)}))_{1\le i,j\le r}
,\]
where $\pmb s_{(\lambda_i,\lambda_j)}=\pmb s_{\lambda_i}\#\pmb s_{\lambda_j}$.
\end{theorem}
\begin{exam}Let $\lambda=(3,2)$ and its outside decomposition $(\theta_1,\theta_2)$ be depicted as 
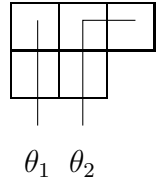
\begin{center}
\begin{figure}[h]
\begin{tikzpicture}
      \node at (3.5,1) {$\ydiagram{3,2}$};
    \node at (2.9,-0.5) {$\theta_1$};
        \node at (3.5,-0.5) {$\theta_2$};
 \draw (2.9,0)--(2.9,1.4);
 \draw (3.5,0)--(3.5,1.4)--(4.2,1.4);
\end{tikzpicture}.
\end{figure}
\end{center}
Then, if $(a_{j-i})=(s_{ij})\in T^{\rm diag}(\lambda,\mathbb C)$, we obtain
\begin{align*}
    \zeta_\lambda^{{\rm so},N}(\pmb s)=&{\rm det}\begin{pmatrix}
\zeta_{\theta_1}^{{\rm so},N}\left(\ytableaushort{{a_{0}},{a_1}}\right)&\zeta_{\theta_1\#\theta_2}^{{\rm so},N}\left(\ytableaushort{{a_{0}}}\right)\\
\zeta_{\theta_2\#\theta_1}^{{\rm so},N}\left(\ytableaushort{{a_1}{a_2},{a_{0}},{a_{-1}}}\right)&\zeta_{\theta_2}^{{\rm so},N}\left(\ytableaushort{{a_1}{a_2},{a_0}}\right)
\end{pmatrix}\\
=&\zeta_{\theta_1}^{{\rm so},N}\left(\ytableaushort{{a_{0}},{a_1}}\right)\zeta_{\theta_2}^{{\rm so},N}\left(\ytableaushort{{a_1}{a_2},{a_0}}\right)-\zeta_{\theta_1\#\theta_2}^{{\rm so},N}\left(\ytableaushort{{a_{0}}}\right)\zeta_{\theta_2\#\theta_1}^{{\rm so},N}\left(\ytableaushort{{a_1}{a_2},{a_{0}},{a_{-1}}}\right).
\end{align*}
\end{exam}
\section{Decomposition of Symplectic and Orthogonal multiple zeta functions}
\label{decompsec}
In this section, we write a symplectic and an orthogonal multiple zeta function as a linear combination of the truncated multiple zeta functions. By the Inclusion-Exclusion principle, we may find the following decompositions.
\begin{theorem}
For any positive integer $N$ and $s_i\in \mathbb C$, we have
\begin{align*}
    \zeta_{(\{1\}^r)}^{{\rm sp},N}\left(\ytableaushort{{s_1},{\vdots},{s_{r}}}\right)=&\sum_{\rm sign}\zeta^N(\pm s_1,\ldots,\pm s_r)\\
    &+\sum_{i=1}^{r-1}\sum_{\rm sign}\zeta^N(\pm s_1,\ldots,\pm s_{i-1},s_{i}-s_{i+1},\pm s_{i+2},\ldots,\pm s_r)\\
    &-\sum_{i=1}^{r-1}\sum_{\rm sign}\left(\prod_{j=1}^{i-1}j^{\pm s_j}\right)\frac{i^{s_{i+1}}}{i^{s_{i}}}\zeta^N(\{0\}^i,\pm s_{i+2},\ldots,\pm s_r),
\end{align*}
where $\displaystyle{\sum_{\rm sign}}$ means the summation over all cases of plus-minus signs. 
\end{theorem}
\begin{exam}[$r\le2$]
For any positive integer $N$ and $a,b\in \mathbb C$, we have
\begin{align*}
\zeta_{(1)}^{{\rm sp},N}\left(\ytableaushort{a}\right)&=\zeta^N(a)+\zeta^{N}(-a),\\
    \zeta_{(1,1)}^{{\rm sp},N}\left(\ytableaushort{a,b}\right)&=\zeta^N(a,b)+\zeta^{N}(-a,-b)+\zeta^N(-a,b)+\zeta^{N}(a,-b)+\zeta^N(a-b)-1.
\end{align*}
\end{exam}
We may apply a similar argument to the orthogonal Schur multiple zeta functions.
\begin{theorem}
Let $r$ be an integer greater than $1$.
For any positive integer $N$ and $s_i\in \mathbb C$, we have
\begin{align*}
    \zeta_{(\{1\}^r)}^{{\rm so},N}\left(\ytableaushort{{s_1},{\vdots},{s_{r}}}\right)=&\sum_{R=r-1}^r\sum_{\rm sign}\zeta^N(\pm s_1,\ldots,\pm s_R)\\
    &+\sum_{R=r-1}\sum_{i=1}^{R-1}\sum_{\rm sign}\zeta^N(\pm s_1,\ldots,\pm s_{i-1},s_{i}-s_{i+1},\pm s_{i+2},\ldots,\pm s_R)\\
    &-\sum_{R=r-1}\sum_{i=1}^{R-1}\sum_{\rm sign}\left(\prod_{j=1}^{i-1}j^{\pm s_j}\right)\frac{i^{s_{i+1}}}{i^{s_{i}}}\zeta^N(\{0\}^i,\pm s_{i+2},\ldots,\pm s_R),
\end{align*}
where $\displaystyle{\sum_{\rm sign}}$ means the summation over all cases of plus-minus signs. 
\end{theorem}
\begin{exam}[$r\le2$]
For any positive integer $N$ and $a,b\in \mathbb C$, we have
\begin{align*}
\zeta_{(1)}^{{\rm so},N}\left(\ytableaushort{a}\right)=&\zeta^N(a)+\zeta^{N}(-a)+1,\\
    \zeta_{(1,1)}^{{\rm so},N}\left(\ytableaushort{a,b}\right)=&\zeta^N(a,b)+\zeta^{N}(-a,-b)+\zeta^N(-a,b)+\zeta^{N}(a,-b)+\zeta^N(a-b)-1\\
    &+\zeta^N(a)+\zeta^{N}(-a).
\end{align*}
\end{exam}
If we use the decompositions by rows as an outside decomposition of $\lambda/\mu$, then for any $\pmb s\in T^{\rm diag}(\lambda/\mu,\mathbb C)$, $\zeta_{\lambda/\mu}^{{\rm sp},N}(\pmb s)$ and $\zeta_{\lambda/\mu}^{{\rm so},N}(\pmb s)$ look like decomposed into a sum of $\zeta_{(\{1\}^r)}^{{\rm sp},N}$ and $\zeta_{(\{1\}^r)}^{{\rm so},N}$, respectively.
As in Remark \ref{contentrem}, we note that the outside decomposition and operation $\theta_i\#\theta_j$ keeps the content and there may be two different function associated with the same shape $\lambda=(\{1\}^r)$ and the same variable $\pmb s=(s_{ij})$.

Similarly, we find the following results, which give a decomposition of symplectic zeta function into a sum of truncated multiple zeta star functions.
\begin{theorem}
For any positive integer $N$ and $s_i\in \mathbb C$, we have
\begin{align*}
    \zeta_{(r)}^{{\rm sp},N}\left(\ytableaushort{{s_1}{\cdots}{s_{r}}}\right)=&\sum_{\rm sign}\sum_{\pmb \ell}(-1)^{r-{\rm dep}(\pmb\ell)}{\zetas}^N(\pmb \ell_r),
    \intertext{and for $r\ge2$}
    \zeta_{(r)}^{{\rm so},N}\left(\ytableaushort{{s_1}{\cdots}{s_{r}}}\right)=&\sum_{R=r-1}^r\sum_{\rm sign}\sum_{\pmb \ell}(-1)^{R-{\rm dep}(\pmb\ell)}{\zetas}^N(\pmb \ell_R),
\end{align*}
where $\displaystyle{\sum_{\rm sign}}$ means the summation over all cases of plus-minus signs and $\pmb \ell$ runs over all indices of the form
$\pmb \ell_R = (\pm s_{1}\square \pm s_{2}\square\cdots\square \pm s_{R})$
in which each $\square$ is filled by the comma $,$ or the plus $+$. If $\square=+$ then $\pm s_j\square\pm s_{j+1}$ is assigned $s_{j+1}-s_{j}$ and the square is not filled consecutive plus signs $+$.
\end{theorem}
\begin{exam}[$r\le2$]
For any positive integer $N$ and $a,b\in \mathbb C$, we have
\begin{align*}
\zeta_{(1)}^{{\rm sp},N}\left(\ytableaushort{a}\right)=&{\zetas}^N(a)+{\zetas}^N(-a),\\
    \zeta_{(2)}^{{\rm sp},N}\left(\ytableaushort{ab}\right)=&{\zetas}^N(a,b)+{\zetas}^N(-a,-b)+{\zetas}^N(-a,b)+{\zetas}^N(a,-b)-{\zetas}^N(b-a),\\
    \zeta_{(1)}^{{\rm so},N}\left(\ytableaushort{a}\right)=&{\zetas}^N(a)+{\zetas}^N(-a)+1,\\
    \zeta_{(2)}^{{\rm so},N}\left(\ytableaushort{ab}\right)=&{\zetas}^N(a,b)+{\zetas}^N(-a,-b)+{\zetas}^N(-a,b)+{\zetas}^N(a,-b)-{\zetas}^N(b-a)\\
    &+{\zetas}^N(a)+{\zetas}^N(-a).
\end{align*}
\end{exam}

\section{Schur quasi-symmetric functions}
\label{sec:Sqsf}

 We here investigate Schur multiple zeta functions from the viewpoint of the quasi-symmetric functions introduced by Gessel \cite{G} (See. \cite{npy}).

\subsection{Quasi-symmetric functions}
 
 Let ${\pmb t}=(t_1,t_2,\ldots)$ be variables and
 $\mathfrak{P}$ a subalgebra of $\mathbb{Z}[\![t_1,t_2,\ldots\,]\!]$ consisting of all formal power series with integer coefficients of bounded degree.
 We call $p=p({\pmb t})\in \mathfrak{P}$ a {\it quasi-symmetric function}
 if the coefficient of $t^{\gamma_1}_{k_1}t^{\gamma_2}_{k_2}\cdots t^{\gamma_n}_{k_n}$ of $p$ is the same as
 that of $t^{\gamma_1}_{h_1}t^{\gamma_2}_{h_2}\cdots t^{\gamma_l}_{h_n}$ of $p$ whenever $k_1<k_2<\cdots <k_n$ and $h_1<h_2<\cdots <h_n$.
 The algebra of all quasi-symmetric functions is denoted by $\mathrm{Qsym}$. 
 For a composition ${\pmb \gamma}=(\gamma_1,\gamma_2,\ldots,\gamma_n)$ of a positive integer,
 define the {\it monomial quasi-symmetric function} $M_{\pmb \gamma}$ and 
 the {\it essential quasi-symmetric function} $E_{\pmb \gamma}$ respectively by
\[
 M_{\pmb \gamma}
=\sum_{m_1<m_2<\cdots<m_n}t_{m_1}^{\gamma_1}t_{m_2}^{\gamma_2}\cdots t^{\gamma_n}_{m_n},
\quad
 E_{\pmb \gamma}
=\sum_{m_1\le m_2\le \cdots \le m_n}t_{m_1}^{\gamma_1}t_{m_2}^{\gamma_2}\cdots t^{\gamma_n}_{m_n}.
\]
 We know that these respectively form the integral basis of $\mathrm{Qsym}$.
 Notice that 
\begin{equation}
\label{for:EM}
 E_{\pmb \gamma}
=\sum_{{\pmb \delta} \,\preceq\, {\pmb \gamma}}M_{\pmb \delta}, \quad 
 M_{\pmb \gamma}
=\sum_{{\pmb \delta} \,\preceq\, {\pmb \gamma}}(-1)^{n-\ell({\pmb \delta})}E_{\pmb \delta}.
\end{equation}
 A relation between the multiple zeta values and the quasi-symmetric functions is studied by Hoffman \cite{H2}
 (remark that the notations of multiple zeta (star) function in \cite{H2} are different from ours;
 they are $\zeta(s_n,s_{n-1},\ldots,s_1)$ and $\zeta^{\star}(s_n,s_{n-1},\ldots,s_1)$, respectively, in our notations).  
 Let $\mathfrak{H}=\mathbb{Z}\langle x,y\rangle$ be the noncommutative polynomial algebra over $\mathbb{Z}$.
 We can define a commutative and associative multiplication $\ast$, called a $\ast$-product, on $\mathfrak{H}$.
 We call $(\mathfrak{H},\ast)$ the (integral) harmonic algebra.   
 Let $\mathfrak{H}^{1}=\mathbb{Z}1+y\mathfrak{H}$, which is a subalgebra of $\mathfrak{H}$.
 Notice that every $w\in \mathfrak{H}^{1}$ can be written as 
 an integral linear combination of $z_{\gamma_1}z_{\gamma_2}\cdots z_{\gamma_n}$ where $z_{\gamma}=yx^{\gamma-1}$ for $\gamma\in\mathbb{N}$.
 For each $N\in \mathbb{N}$,
 define the homomorphism $\phi_N:\mathfrak{H}^{1}\to \mathbb{Z}[t_1,t_2,\ldots,t_N]$ by $\phi_N(1)=1$ and 
\[
 \phi_N(z_{\gamma_1}z_{\gamma_2}\cdots z_{\gamma_n})
=
\begin{cases}
 \displaystyle{\sum_{m_1<m_2<\cdots<m_n\le N}t_{m_1}^{\gamma_1}t_{m_2}^{\gamma_2}\cdots t^{\gamma_n}_{m_n}} & n\le N, \\
 0 & \text{otherwise},
\end{cases}
\]
 and extend it additively to $\mathfrak{H}^{1}$.
 There is a unique homomorphism $\phi:\mathfrak{H}^{1}\to \mathfrak{P}$ such that $\pi_N\phi=\phi_N$
 where $\pi_N$ is the natural projection from $\mathfrak{P}$ to $\mathbb{Z}[t_1,t_2,\ldots,t_N]$.
 We have $\phi(z_{\gamma_1}z_{\gamma_2}\cdots z_{\gamma_n})=M_{(\gamma_1,\gamma_2,\ldots,\gamma_n)}$.
 Moreover, as is described in \cite{H2}, $\phi$ is an isomorphism between $\mathfrak{H}^1$ and $\mathrm{Qsym}$.

 Let $e$ be the function sending $t_i$ to $\frac{1}{i}$. 
 Moreover, define $\rho_N:\mathfrak{H}^{1}\to\mathbb{R}$ by $\rho_N=e\phi_N$.
 For a composition ${\pmb \gamma}$, we have 
\[
 \rho_N\phi^{-1}(M_{\pmb \gamma})=\zeta^{N}({\pmb \gamma}), \quad 
 \rho_N\phi^{-1}(E_{\pmb \gamma})=\zeta^{N\star}({\pmb \gamma}). 
\]
 Define the map $\rho:\mathfrak{H}^1\to\mathbb{R}^{\mathbb{N}}$ by 
 $\rho(w)=(\rho_N(w))_{N\in\mathbb{N}}$ for $w\in\mathfrak{H}^{1}$. 
 Notice that if $w\in \mathfrak{H}^0=\mathbb{Z}1+y\mathfrak{H}x$, which is a subalgebra of $\mathfrak{H}^1$,
 then we may understand that $\rho(w)=\lim_{N\to\infty}\rho_N(w)\in\mathbb{R}$.
 In particular, for a composition ${\pmb \gamma}=(\gamma_1,\gamma_2,\ldots,\gamma_n)$ with $\gamma_n\ge 2$,
 we have 
\begin{equation}
\label{for:rho}
 \rho\phi^{-1}(M_{\pmb \gamma})
=\zeta({\pmb \gamma}), \quad
 \rho\phi^{-1}(E_{\pmb \gamma})
=\zeta^{\star}({\pmb \gamma}).
\end{equation}

\subsection{Schur $P$- and $Q$-type quasi-symmetric functions}

 Now, one easily reaches the definition of the following {\it Schur $P$- and $Q$-type quasi-symmetric functions}. 
 For strict partitions $\lambda$ and $\mu$, and ${\pmb s}=(s_{ij})\in ST(\lambda/\mu,\mathbb{C})$,
 we define Schur $P$- and $Q$-type quasi-symmetric functions associated with $\lambda/\mu$ by
\begin{equation}
\label{def:skewQSPMZ}
 S_{\lambda/\mu}^P({\pmb s})
=\sum_{M\in PSST(\lambda/\mu)}\displaystyle{\prod_{(i, j)\in SD(\lambda)}|m_{ij}|^{s_{ij}}},
\end{equation}
and
\begin{equation}
\label{def:skewQSQMZ}
 S_{\lambda/\mu}^Q({\pmb s})
=\sum_{M\in QSST(\lambda/\mu)}\displaystyle{\prod_{(i, j)\in SD(\lambda)}|m_{ij}|^{s_{ij}}}.
\end{equation}
\begin{theorem}
Let $\lambda=(\lambda_1,\ldots,\lambda_{r})$, $\mu=(\mu_1,\ldots,\mu_{s})$ be strict partitions into with $\lambda_i\ge0$ and $2|r+s$. 
Then for $\pmb s\in ST^{\rm diag}(\lambda/\mu,\mathbb C)$,
\[S_{\lambda/\mu}^Q(\pmb s)={\rm pf}\begin{pmatrix}
M_\lambda&H_{\lambda,\mu}\\
0&0
\end{pmatrix}
,\]
where $M_\lambda=(a_{ij})$ is an $r\times r$ upper triangular matrix with \[a_{ij}=S_{(\lambda_i,\lambda_j)}^Q(\pmb s_{(\lambda_i,\lambda_j)}),\]
\[\pmb s_{(\lambda_i,\lambda_j)}=\ytableaushort{{s_{ii}}{\cdots}{\cdots}{\cdots}{s_{it_i}},{\none}{s_{jj}}{\cdots}{s_{jt_j}}},\]
 where $t_i=i+\lambda_i-1$ and $H_\lambda=(b_{ij})$ is an $r\times s$ matrix with \[b_{ij}=S_{(\lambda_i-\mu_s-j+1)}^Q(s_{i(i+j+\mu_s-1)},\ldots,s_{it_i}).\]
\end{theorem}
\begin{theorem}[cf. {\cite[Theorem 4.3]{fk},\cite[Theorem 1.4]{h96}}]
Let $\lambda$ and $\mu$ be strict partitions with $\mu\le\lambda$.
Let $\theta =(\theta_1, \theta_2,\ldots, \theta_k, \theta_{k+1},\ldots,\theta_r)$ be an outside decomposition of $SD(\lambda/\mu)$, where $\theta_p$ includes a box on the main
diagonal of $SD(\lambda/\mu)$ for $1\le p\le k$ and $\theta_p$ does not for $k+1\le p\le r$. If $k$ is odd, we replace $\theta$ by $(\emptyset,\theta_1,\ldots,\theta_r)$.  
Then the Schur $Q$-type quasi-symmetric functions satisfy the identity
\[S_{\lambda/\mu}^Q(\pmb s)={\rm pf}\begin{pmatrix}
S_{(\overline{\theta}_p,\overline{\theta}_q)}^Q(\pmb s_{(\overline{\theta}_p,\overline{\theta}_q)})&S_{\theta_i\#\theta_{r+k+1-j}}^Q(\pmb s_{\theta_i\#\theta_{r+k+1-j}})\\
-{}^t(S_{\theta_i\#\theta_{r+k+1-j}}^Q(\pmb s_{\theta_i\#\theta_{r+k+1-j}}))&0
\end{pmatrix}\]
with $1 \le p, q \le k$ and $k + 1\le j \le r$. Here
\[S_{(\overline{\theta}_p,\overline{\theta}_q)}^Q(\pmb s_{(\overline{\theta}_p,\overline{\theta}_q)})= -S_{(\overline{\theta}_q,\overline{\theta}_p)}^Q(\pmb s_{(\overline{\theta}_q,\overline{\theta}_p)})\]
and $S_{(\overline{\theta}_p,\overline{\theta}_p)}^Q(\pmb s_{(\overline{\theta}_p,\overline{\theta}_q)})=0$.
\end{theorem}

\subsection{Symplectic type and Orthogonal type quasi-symmetric functions}

 Similarly, we define the following {\it symplectic quasi-symmetric functions} and {\it orthogonal quasi-symmetric functions}. 
 For partitions $\lambda$ and $\mu$, and ${\pmb s}=(s_{ij})\in T(\lambda/\mu,\mathbb{C})$, we define symplectic quasi-symmetric functions and orthogonal quasi-symmetric functions associated with $\lambda/\mu$ by
\begin{equation}
\label{def:QSPMZ}
 S_{\lambda/\mu}^{{\rm sp},N}({\pmb s})
=\sum_{M\in SP_N(\lambda/\mu)}\displaystyle{\prod_{(i, j)\in D(\lambda)}|m_{ij}|^{s_{ij}}},
\end{equation}
and
\begin{equation}
\label{def:QOMZ}
 S_{\lambda/\mu}^{{\rm so},N}({\pmb s})
=\sum_{M\in SO_N(\lambda/\mu)}\displaystyle{\prod_{(i, j)\in D(\lambda)}|m_{ij}|^{s_{ij}}}.
\end{equation}
\begin{theorem}
Let $\lambda=(\lambda_1,\ldots,\lambda_{r})$, $\mu=(\mu_1,\ldots,\mu_{s})$ be partitions. 
Then for $\pmb s\in T^{\rm diag}(\lambda/\mu,\mathbb C)$ and any outside decomposition $(\theta_1,\ldots,\theta_r)$ of $\lambda/\mu$,
\[S_{\lambda/\mu}^{{\rm sp},N}(\pmb s)={\rm det}(S_{\theta_i\#\theta_j}^{{\rm sp},N}(\pmb s_{(\lambda_i,\lambda_j)}))_{1\le i,j\le r}
,\]
where $\pmb s_{(\lambda_i,\lambda_j)}=\pmb s_{\lambda_i}\#\pmb s_{\lambda_j}$.
\end{theorem}
\begin{theorem}
Let $\lambda=(\lambda_1,\ldots,\lambda_{r})$, $\mu=(\mu_1,\ldots,\mu_{s})$ be partitions. 
Then for $\pmb s\in T^{\rm diag}(\lambda/\mu,\mathbb C)$ and any outside decomposition $(\theta_1,\ldots,\theta_r)$ of $\lambda/\mu$,
\[S_{\lambda/\mu}^{{\rm so},N}(\pmb s)={\rm det}(S_{\theta_i\#\theta_j}^{{\rm so},N}(\pmb s_{(\lambda_i,\lambda_j)}))_{1\le i,j\le r}
,\]
where $\pmb s_{(\lambda_i,\lambda_j)}=\pmb s_{\lambda_i}\#\pmb s_{\lambda_j}$.
\end{theorem}

\section*{Acknowledgement}
The authors would like to thank Professor Takeshi Ikeda and Professor Soichi Okada for their helpful comments.
This work was supported by Grant-in-Aid for Scientific Research (C) (Grant Number: JP22K03274) and Grant-in-Aid for Early-Career Scientists (Grant Number: JP22K13900).

\end{document}